\documentclass[11pt]{amsart}

\usepackage{epigamath}

\usepackage[english]{babel}
\usepackage[all]{xy}
\usepackage{mathabx}
\usepackage{tikz}
\usepackage{tikz-cd}

\numberwithin{equation}{section}

\usepackage{enumitem}

\newtheorem{theorem}{Theorem}[section]
\newtheorem{corollary}[theorem]{Corollary}
\newtheorem{lemma}[theorem]{Lemma}
\newtheorem{proposition}[theorem]{Proposition}
\newtheorem{proposition-definition}[theorem]{Proposition-Definition}

\theoremstyle{definition}

\newtheorem{remark}[theorem]{Remark}

\DeclareMathOperator{\Jac}{Jac}

\newcommand{\BBB}{\mathrm{(BBB)}}
\newcommand{\BAA}{\mathrm{(BAA)}}

\newcommand{\Tot}{\mathrm{Tot}}

\newcommand{\supp}{\mathrm{supp}}
\newcommand{\End}{\mathrm{End}}
\newcommand{\Ext}{\mathrm{Ext}}
\newcommand{\pr}{\mathrm{pr}}

\newcommand{\st}{\mathrm{st}}

\newcommand{\gr}{\mathrm{gr}}
\newcommand{\Id}{\mathrm{1}}
\newcommand{\im}{\mathrm{Im} \,}
\newcommand{\rk}{\mathrm{rk}}
\newcommand{\Tors}{\mathrm{Tors}}

\newcommand{\Coh}{\mathrm{Coh}}
\newcommand{\tr}{\mathrm{tr}}

\newcommand{\ii}{\,\mathrm{i}\,}

\newcommand{\GL}{\mathrm{GL}}

\newcommand{\U}{\mathrm{U}}

\newcommand{\coker}{\mathrm{coker}}

\newcommand{\Un}{\mathrm{U}}
\newcommand{\WIT}{\mathrm{Wit}}
\newcommand{\WWIT}{\mathbf{Wit}}
\newcommand{\wWIT}{\mathfrak{Wit}}

\newcommand{\Aa}{\mathcal{A}}
\newcommand{\Bb}{\mathcal{B}}

\newcommand{\Dd}{\mathcal{D}}
\newcommand{\Ee}{\mathcal{E}}
\newcommand{\Ff}{\mathcal{F}}
\newcommand{\Gg}{\mathcal{G}}

\newcommand{\Ll}{\mathcal{L}}

\newcommand{\Oo}{\mathcal{O}}
\newcommand{\Pp}{\mathcal{P}}

\newcommand{\Ss}{\mathcal{S}}

\renewcommand{\D}{\mathbf{D}}
\newcommand{\E}{\mathbf{E}}
\newcommand{\F}{\mathbf{F}}

\renewcommand{\P}{\mathbf{P}}

\renewcommand{\U}{\mathbf{U}}

\newcommand{\M}{\mathrm{M}}

\newcommand{\Ddd}{\mathfrak{D}}
\newcommand{\Eee}{\mathfrak{E}}

\newcommand{\Ggg}{\mathfrak{G}}

\newcommand{\DDd}{\mathscr{D}}

\newcommand{\FFf}{\mathscr{F}}

\newcommand{\ol}[1]{\overline{#1}}
\newcommand{\ul}[1]{\underline{#1}}

\newcommand{\wt}[1]{\widetilde{#1}}

\newcommand{\fF}{\mathfrak{f}}

\newcommand{\iI}{\mathfrak{i}}

\newcommand{\CC}{\mathbb{C}}

\newcommand{\EE}{\mathbb{E}}
\newcommand{\RR}{\mathbb{R}}

\newcommand{\HH}{\mathbb{H}}

\newcommand{\morph}[6]{\begin{array}{cccc} #6: & #1  & \stackrel{#5}{\longrightarrow} &  #2  \\ & #3 & \longmapsto & #4  \end{array}}

\EpigaVolumeYear{6}{2022}
\EpigaArticleNr{4}
\ReceivedOn{June 29, 2020}
\InFinalFormOn{December 14, 2021}
\AcceptedOn{December 31, 2021}

\title{Mirror symmetry for Nahm branes}
\titlemark{Mirror symmetry for Nahm branes}

\author{Emilio Franco}
\address{%
  Centro de An\'alise Matem\'atica,
  Geometria e Sistemas Din\^{a}micos, 
  Instituto Superior T\'ecnico,
  Universidade de Lisboa, 
  Av. Rovisco Pais s/n, 1049-001 Lisboa, Portugal
}
\email{emilio.franco@tecnico.ulisboa.pt}

\author{Marcos Jardim}
\address{Instituto de Matem\'atica, Estat\'istica e Computa\c{c}\~ao Cien\-t\'i\-fi\-ca, Universidade Estadual de Campinas, Rua S\'ergio Buarque de Holanda, 651, 13083-859 Campinas/SP (Brazil)}
\email{jardim@ime.unicamp.br}

\authormark{E.~Franco and M.~Jardim} 

\AbstractInEnglish{%
  The Dirac--Higgs bundle is a hyperholomorphic bundle over the moduli
  space of stable Higgs bundles of coprime rank and degree. We provide
  an algebraic generalization to the case of trivial degree and the
  rank higher than $1$. This allow us to generalize to this case the
  Nahm transform defined by Frejlich and the second named author,
  which, out of a stable Higgs bundle, produces a vector bundle with
  connection over the moduli space of rank $1$ Higgs bundles. By
  performing the higher rank Nahm transform we obtain a
  hyperholomorphic bundle with connection over the moduli space of
  stable Higgs bundles of rank $n$ and degree $0$, twisted by the
  gerbe of liftings of the projective universal bundle.

  Such hyperholomorphic vector bundles over the moduli space of stable
  Higgs bundles can be seen, in the physicist's language, as
  $\BBB$-branes twisted by the above mentioned gerbe. We refer to these
  objects as Nahm branes. Finally, we study the behaviour of Nahm branes
  under Fourier--Mukai transform over the smooth locus of the Hitchin
  fibration, checking that the resulting objects are supported on a
  Lagrangian multisection of the Hitchin fibration, so they describe
  partial data of $\BAA$-branes.}



\acknowledgement{EF is currently supported by the Scientific
  Employment Stimulus program, fellowship reference CEECIND/04153/2017,
  funded by FCT (Portugal). He is partially supported by FCT/Portugal
  through CAMGSD, IST-ID, projects UIDB/04459/2020 and UIDP/04459/2020.
  He was previously supported by the project PTDC/MAT- GEO/2823/2014
  also funded by FCT (Portugal) with national funds, and by the FAPESP
  postdoctoral grant number 2012/16356-6. MJ is partially supported by
  the CNPq grant number 302889/2018-3, the FAPESP Thematic Project
  2018/21391-1; the present work was also partially funded by the FAPESP
  grant number 2016/03759-6, which supported the second named author
  during his visit to University of Edinburgh in 2017, where part of
  this research was carried out.
}

\begin{document}

%%%%%%%%%%%%%%%%%%%%%%%%%%%%%%%
% Title page
%%%%%%%%%%%%%%%%%%%%%%%%%%%%%%%

\removeabove{.5cm}
\removebetween{.5cm}
\removebelow{.5cm}

\maketitle

\begin{prelims}

\DisplayAbstractInEnglish

\bigskip

\DisplayKeyWords

\medskip

\DisplayMSCclass
\end{prelims}

%%%%%%%%%%%%%%%%%%%%%
% Table of Contents
%%%%%%%%%%%%%%%%%%%%%

\newpage

\setcounter{tocdepth}{1}

\tableofcontents

%%%%%%%%%%%%%%%%%%%%%
% Content begins here
%%%%%%%%%%%%%%%%%%%%%

\section{Introduction}

\subsection{Context}

The Nahm transform was originally described as a geometrical correspondence between solutions of the self-duality Yang--Mills equations (also known as instantons) in $\RR^4$ which are invariant under dual groups of translations \cite{ADHM, Nahm_1, Nahm_2, Hitchin_construction_monopoles, corrigan&goddard, Braam&vanBaal, Nakajima, jardim_2}. In \cite{jardim_survey}, the second-named author reviewed the Nahm transform in several situations and gave an interpretation as a nonlinear version of the Fourier transform which, given a family of self-dual connections over a spin four-manifold with non-negative scalar curvature, produces a vector bundle with connection over the parametrizing space of the family. Such bundle is constructed by considering, at each point of the parametrizing space, the cokernel of the associated Dirac operator. The connection is hyperk\"ahler whenever both varieties, the base manifold and the parametrizing space, are hyperk\"ahler.

The study of instantons that are invariant under translations in two directions led Hitchin to introduce Higgs bundles in \cite{hitchin-self} as solutions of the dimensional reduction to a Riemann surface of the self-dual connections in $4$ dimensions. It turns out that the moduli space $\M$ of Higgs bundles has a very rich geometry; in particular, it can be constructed as a hyperk\"ahler quotient in the context of gauge theory \cite{hitchin-self, simpson1, simpson2, donaldson, corlette} inheriting a hyperk\"ahler structure. When the rank and the degree are coprime, all semistable Higgs bundles are stable, and Hausel \cite{hausel} showed that a universal bundle exists. In degree $0$, however, there is no universal bundle over the stable locus $\M^\st$ or even over any other Zariski open set of the moduli space (see Ramanan \cite{ramanan} and Drezet and Narasimhan \cite{drezet&narasimhan} for a proof in the case of vector bundles that extends naturally to Higgs bundles). There exists nonetheless a local universal bundle in the \'etale topology of the moduli space of stable Higgs bundles $\M^\st$, as indicated by Simpson \cite{simpson2}.

Another important feature of the moduli space $\M$ of stable Higgs bundles with coprime rank and degree is the existence of the so called \emph{Dirac--Higgs bundle}, originally considered by Hitchin and studied in detail by Hausel in \cite{hausel}. Later, Blaavand \cite{blaavand} extended the construction of the Dirac--Higgs bundle to the moduli space of parabolic Higgs bundles.

One way to describe the Dirac--Higgs bundle is as the hyperholomorphic bundle on $\M$ obtained via the Nahm transform, as defined in \cite{frejlich&jardim}, associated to the universal bundle.  To be more precise, the Nahm transform of a Higgs bundle is defined by considering the index bundle associated to the family obtained by twisting the original Higgs bundle with the universal family of rank 1 Higgs bundles. This transform underlies the Fourier--Mukai transform for Higgs bundles defined by Bonsdorff \cite{bonsdorff_1}, which also equipped it with an autodual connection \cite{bonsdorff_2}.

Another interesting feature of the moduli space $\M$ of Higgs bundles is that it admits a fibration $\M \to B$ over a vector space, becoming an algebraically completely integrable system \cite{hitchin_duke} which is known as the {\it Hitchin system}. The notion of Higgs bundles generalizes naturally to any structure group $G$. It was shown in \cite{hausel&thaddeus, donagi&gaitsgory, donagi&pantev} that Hitchin systems for Langlands dual structure groups, $G$ and $G^L$, are dual, satisfying thereby the requirements of being \emph{Strominger--Yau--Zaslow (SYZ) mirror partners} \cite{SYZ}, which allows for the identification of T-duality with mirror symmetry between them. Since the group $G = \GL(n,\CC)$ is Langlands self-dual, we obtain a self-dual Hitchin system in this case, which is the one that we study in this paper.

The rich geometry of the moduli space of Higgs bundles $\M$ makes it an object of interest for theoretical physics. In \cite{vafa, strominger} it was shown that the dimensional reduction of an $N = 4$ Super Yang--Mills theory in $4$ dimensions gives a $2$ dimensional sigma model with hyperk\"ahler target $\M$, and, hence, S-duality in the former becomes T-duality (mirror symmetry) in the latter. This was the starting point for the ground-breaking article by Kapustin and Witten \cite{kapustin&witten} (see also \cite{witten}), where they relate the Geometric Langlands Conjecture and S-duality in the original $N = 4$ super Yang--Mills theory. Following Kapustin and Witten, a $\BBB$-brane is a pair consisting of a hyperk\"ahler submanifold and a hyperholomorphic vector bundle. Similarly, a $\BAA$-brane is given by a submanifold which is complex Lagrangian with respect to the first K\"ahler structure, and a flat vector bundle. In String Theory, branes are geometrical objects that encode the Dirichlet boundary conditions, and mirror symmetry \cite{kapustin&witten, witten} predicts a 1-1 correspondence between $\BBB$-branes on the moduli space of $G$-Higgs bundles and $\BAA$-branes on its $G^L$ counterpart.

Motivated by this context, many authors have studied branes in moduli spaces of Higgs bundles, see for instance \cite{hitchin_char, hitchin_spinors, gaiotto, BS1,BS2, BS3, BGP, heller&schaposnik, biswas&calvo&franco&garciaP, branco, Borel, FGOP, HMDP} and \cite{AFES} for a survey on this topic, and the geometry of these objects has been intensively described \cite{biswas&schaffhauser, biswas&garciaP&hurtubise1, biswas&garciaP&hurtubise2, garciaP&wilkins, baird1, baird2, schottky, schaffhauser}. More generally, due to their intrinsic geometric interest, one can also study hyperk\"ahler and complex Lagrangian submanifolds on other classes of hyperk\"ahler manifolds, like quiver varieties \cite{FJMa,HS}, and moduli spaces of stable sheaves on symplectic surfaces \cite{FJMe}.

\subsection{Our constructions}

In this paper we generalize the constructions of the Dirac--Higgs bundle and the Nahm transform of a stable Higgs bundle to the case of trivial degree and rank higher than $1$. This provides a class of (space filling) $\BBB$-branes that we transform under Fourier--Mukai, obtaining a partial description of the mirror dual $\BAA$-brane. 

In the case of coprime rank and degree, the universal Higgs bundle plays a central role in the construction of the Dirac--Higgs bundle. In our case, however, there is no universal Higgs bundle at hand, not even locally. To surpass this obstacle, we consider the gerbe of liftings of the projective universal bundle\footnote{We are grateful to the anonymous referee for suggesting this approach.}, and we introduce the notions of sheaves twisted and shifted by such gerbe. We can then construct the Dirac--Higgs bundle as a vector bundle twisted by our gerbe, showing that it is equipped with a hyperholomorphic connection. The techniques used in the construction of the Dirac--Higgs bundle can then be applied to define the Nahm transform of a stable Higgs bundle which is, again, a bundle twisted by our gerbe and equipped with a hyperholomorphic connection. This constitutes a family of $\BBB$-branes (one for each stable Higgs bundle) which we call \emph{Nahm branes}.

The second step is to study the behaviour of these Nahm branes under mirror symmetry. We work over the smooth locus of the Hitchin fibration, where mirror symmetry is expected to be realized via Fourier--Mukai transform. We check that the transformed sheaf is supported on a complex Lagrangian multisection of the Hitchin fibration. This is part of the data of a $\BAA$-brane, providing evidence for the existence of a correspondence between $\BBB$ and $\BAA$-branes conjectured by Kapustin and Witten.

\subsection{Organization of the paper}

In Sections \ref{sc Mm hyperkahler} and \ref{sc Hitchin fibration} we review the properties of the theory of Higgs bundles and the Hitchin system. In Section \ref{sc mirror} we survey, after restricting ourselves to the smooth locus of the Hitchin fibration, the Fourier--Mukai transform, which is expected to realize mirror symmetry in this context. In Section \ref{sc gerbes} we review gerbes and the notions of sheaves twisted by a gerbe, we also describe an example of gerbe over the moduli space of Higgs bundles of particular importance for us, the gerbe of liftings of the projective universal bundle. In Section \ref{sc Dirac-Higgs} we review the Dirac--Higgs bundle in general, and we provide an algebraic construction over the moduli space of Higgs bundles with $0$ degree as a bundle twisted by the gerbe of liftings. In Section \ref{sc tensorization} we study the behaviour of spectral data of Higgs bundles under tensorization, crucial for understanding our generalization of the Nahm transform of a stable Higgs bundle to rank higher than $1$, which we achieve in Section \ref{sc high rank Nahm}, thereby constructing Nahm $\BBB$-branes. In Section \ref{sc FM for twisted vector bundles} we adapt the Fourier--Mukai transform for sheaves twisted by a gerbe, in order to study in Sections \ref{sc FM Dirac-Higgs} and \ref{sc Fourier-Mukai on Nahm transform}, respectively, the transform of the Dirac--Higgs bundle and of the Nahm brane associated to any stable Higgs bundle.

%%%%%%%%%%%%%%%%%%%%%%%%%%%%%%%%%%%%%%%%%%%%%%%%%%%%%%%%%%%%%%%%%%%%%%%%%%%%%%%%%%%%%%%%%%%%%%%%%%%%%%%%%%%%%%%%%%%%%%%%%%%%%%%%%%%%%%%%%%%%%%%%%%%%%%%%%%%%%%%%%%%%%%%%%%%%%%%%%%%%%%%%%%%%%%%%%%%%%%%%%%%%%%

\section{Geometry of the Hitchin system}
\label{sc preliminaries}

%%%%%%%%%%%%%%%%%%%%%%%%%%%%%%%%%%%%%%%%%

\subsection{Non-abelian Hodge theory}\label{sc Mm hyperkahler}

In this section we introduce the moduli space of Higgs bundles, an object with an extremely rich geometry. In particular, it is equipped with a hyperk\"ahler structure.

Let us consider a smooth projective curve $X$ over $\CC$ of genus $g \geq 2$. Denote by $\EE$ the unique up to isomorphism $C^\infty$-bundle of rank $n$ over $X$ and consider a Hermitian metric $h$ on it. Denote by $\Gg$ the Gauge group of unitary automorphism of $\EE$ preserving the metric, and its complexification, $\Gg^\CC$, parametrizing all automorphisms of $\EE$. Recall that $\EE$ equipped with a Dolbeault operator $\ol{\partial}_E$ gives rise to a holomorphic vector bundle $E$. Out of $\ol{\partial}_E$ and the metric $h$, one can construct $\partial_E$ and $\partial_E + \ol{\partial}_E$ is a unitary connection on $E$.

A {\it Higgs pair} of rank $n$ on $X$ is pair $(\overline{\partial}_E, \varphi)$ where $\overline{\partial}_E$ is a Dolbeault operator on $\EE$ fixing an integrable complex structure on it, and $\varphi$ is an element of $\Omega_X^{1,0}(\End(\EE))$. Note that the space $\Aa$ of Higgs pairs is an affine space modeled in the infinite dimensional vector space $\Omega_X^{0,1}(\End(\EE)) \oplus \Omega_X^{1,0}(\End(\EE))$.

A {\it Higgs bundle} over $X$ is a Higgs pair $(\overline{\partial}_E, \varphi)$ satisfying $\overline{\partial}_E \, \varphi = 0$ (hence $\partial_E \, \varphi^* = 0$, where $\varphi^*$ is the adjoint of $\varphi$ with respect to the metric $h$). Equivalently, a Higgs bundle is a pair $\Ee = (E, \varphi)$, where $E$ is a holomorphic vector bundle on $X$, and $\varphi \in H^0(X, \End(E) \otimes K_X)$ is a holomorphic section of the endomorphisms bundle, twisted by the canonical bundle $K_X$. Let us write $\Bb \subset \Aa$ for the $\Gg^\CC$-invariant subset parametrizing Higgs bundles.

Recall from \cite{hitchin-self, simpson1, nitsure} that a Higgs bundle $\Ee$ is said to be \emph{(semi)stable} if every proper, non-trivial, $\varphi$-invariant sub-bundle $F\subset E$ satisfies 
$$ \frac{\deg F}{\rk F} < ~(\leq)~ \frac{\deg E}{\rk E} . $$
In addition, $\Ee$ is \emph{polystable} if it is a direct sum of semistable bundles $\Ee_i=(E_i, \varphi_i)$, all with the same slope $\deg E_i/\rk E_i$. It is possible to construct \cite{hitchin-self, simpson1, simpson2, nitsure} the moduli space,
\[
  \M = \Bb \big/\hspace{-.8ex}\big/ \Gg^\CC,
\]
of rank $n$ and degree $0$ semistable Higgs bundles on $X$. The locus of stable Higgs bundles $\M^\st \subset \M$ is a dense open subset identified with the quotient under $\Gg^\CC$ of $\Bb^{\, \st} \subset \Bb$, the set of stable Higgs bundles. When we need to specify the rank $n$ we shall write $\M_n$ and $\M_n^\st$.

Non-abelian Hodge theory establishes the existence of a homeomorphism \cite{hitchin-self, simpson1, simpson2, donaldson, corlette} between $\M$ and the moduli space of flat connections of rank $n$. This is a consequence of the construction of these moduli spaces as a hyperk\"ahler quotient of the space of Higgs pairs $\Aa$ by the gauge group $\Gg^\CC$ of complex automorphisms of $\EE$. Tangent to any Higgs pair $(\overline{\partial}_E, \varphi)$ we can consider its infinitesimal deformations $\dot{\alpha} \, d\ol{z} \in \Omega_X^{0,1}(\End(\EE))$ and $\dot{\varphi} \, dz \in \Omega_X^{1,0}(\End(\EE))$, with $z$ being a holomorphic coordinate of the base curve. Also, we consider $\dot{\alpha}^* \, dz \in \Omega_X^{1,0}(\End(\EE))$ and $\dot{\varphi}^* \, d\ol{z} \in \Omega_X^{0,1}(\End(\EE))$ to be infinitesimal deformations of $\partial_E$ and $\varphi^*$ respectively. The hyperk\"ahler structure on $\Aa$ is given by the following metric on this space,
\begin{equation}
  \label{eq HB metric}
\widetilde{g} \left ( (\dot{\alpha}_1, \dot{\varphi}_1), (\dot{\alpha}_2, \dot{\varphi}_2) \right ) = \frac{1}{4\pi} \int_X \tr \left ( \dot{\alpha}_1^* \dot{\alpha}_2 + \dot{\alpha}_2^* \dot{\alpha}_1 + \dot{\varphi}_1  \dot{\varphi}_2^* + \dot{\varphi}_2 \dot{\varphi}_1^* \right ) dz \wedge d\ol{z},
\end{equation}
and the complex structures $\widetilde{\Gamma}^1$,
$\widetilde{\Gamma}^2$, and $\widetilde{\Gamma}^3 =
\widetilde{\Gamma}^1 \widetilde{\Gamma}^2$. We denote by
$\widetilde{\omega}^j(\cdot, \cdot)
= \widetilde{g}(\cdot, \widetilde{\Gamma}^j(\cdot))$
the associated K\"ahler forms and consider the symplectic
forms $\widetilde{\Lambda}^j = \widetilde{\omega}^{j+1} + \ii
\widetilde{\omega}^{j-1}$ which are holomorphic for the corresponding
$\wt{\Gamma}^j$. 

From each of the K\"ahler forms $\widetilde{\omega}^j$ one can construct a moment map $\mu^j$. The space of Higgs bundles $\Bb$ can be identified with $(\mu^2)^{-1}(0) \cap (\mu^3)^{-1}(0)$ inside $\Aa$. Then, after an infinite-dimensional version of Kempf--Ness theorem, $\M$ can be identified with the hyperk\"ahler quotient 
\[
  \M \cong
  (\mu^1)^{-1}(0) \cap (\mu^2)^{-1}(0) \cap (\mu^3)^{-1}(0) \big/ \Gg.
\]
The $\Gg$-invariant complex structures $\widetilde{\Gamma}^j$ descend naturally to complex structures $\Gamma^j$ on $\M$. Also, the $2$-forms $\widetilde{\omega}^j$ and $\widetilde{\Lambda}^j$ are $\Gg$-invariant, so they provide naturally the K\"ahler forms $\omega^j$ on $\M$, and the holomorphic symplectic forms $\Lambda^j$.

Given a Higgs bundle $\Ee = (E,\varphi)$, consider the complex
\[
  C^\bullet_{\Ee} :
  \End (E)
  \stackrel{[ \cdot, \varphi]}{\longrightarrow}
  \End (E) \otimes K_X,
\]
which induces the following exact sequence
\begin{center}
\begin{tikzcd}
0\ar[r] & \HH^0(C^\bullet_{\Ee})\ar[r] & H^0(\End(E))\ar[d,phantom,""{coordinate, name=Y}] \ar[r," {[\cdot,\varphi]} "] & H^0(\End(E) \otimes K_X) \arrow[dll,
rounded corners,
to path={ -- ([xshift=2ex]\tikztostart.east)
|- (Y) [near start]\tikztonodes
-| ([xshift=-2ex]\tikztotarget.west)
-- (\tikztotarget)}] & ~ \\
~ & \HH^1(C^\bullet_{\Ee})\ar[r,"\eta"] & H^1(\End(E)) \ar[r,"{[\cdot,\varphi]}"] & H^1(\End(E) \otimes K_X) \ar[r] & \HH^2(C^\bullet_{\Ee}) \ar[r] & 0.
\end{tikzcd}
\end{center}
where $\HH^p(C^\bullet_{\Ee})$ are the hypercohomology groups for the complex $C^\bullet_{\Ee}$. If $\Ee$ is stable, $\HH^0(C^\bullet_{\Ee}) \cong \HH^2(C^\bullet_{\Ee}) \cong \CC$ so $\HH^1(C^\bullet_{\Ee})$ has fixed dimension $2n^2(g-1) + 2$ and, furthermore, it can be identified with the tangent space $T_{\Ee} \M^\st = \HH^1(C^\bullet_{\Ee})$. Then, $\M^\st$ is smooth, and since $\M^\st \subset \M$ is a dense open subset, one has that
\[
\dim \M = 2n^2(g-1) + 2.
\]

Thanks to Serre duality, $\varphi \in H^0(\End(E) \otimes K_X)$ can also be regarded as an element of the dual space $H^1(\End(E))^*$; one can define a 1-form $\theta$ on $\M^{\rm st}$ as the composition of the map $\eta:\HH^1(C^\bullet_\Ee) \to H^1(\End(E))$ and the pairing with $\varphi$, \emph{i.e.} $\theta(v) = \langle \varphi, \eta(v)\rangle$, for each $v \in \HH^1(C^\bullet_\Ee)$. Hence, $d\theta$ is proportional to the holomorphic symplectic form $\Lambda^1$, obtained from $\wt{\Lambda}^1$.

%--------------------------------------------------------------------------------------------------------------------------------------------------------------------------------------------------------------

\subsection{The Hitchin fibration}
\label{sc Hitchin fibration}

Besides its hyperk\"ahler nature, the moduli space of Higgs bundles is equipped with the structure of an algebraically completely integrable system by means of the spectral correspondence that we revise \cite{hitchin_duke, BNR, simpson2} in this section. 

Let $(q_1, \dots, q_n)$ be a basis of $\GL(n,\CC)$-invariant polynomials  with $\deg(q_i) = i$. The {\it Hitchin fibration} is the dominant morphism
$$
\morph{\M}{B := \bigoplus_{i = 1}^n H^0(X,K_X^{\otimes i})}{(E,\varphi)}{\left ( q_1(\varphi), \dots, q_n(\varphi) \right ),}{}{h}
$$
and we refer to $B$ as the {\it Hitchin base}.
Consider the total space $\Tot(K_X)$ of the canonical bundle, and the obvious algebraic surjection $p: \Tot(K_X) \to X$; let $\lambda$ be the tautological section of the pullback bundle $p^*K_X \to \Tot(K_X)$. Given an element $b = (b_1, \dots, b_n) \in B$ we construct the associated {\it spectral curve} $S_b \subset \Tot(K_X)$ by considering the vanishing locus of the section
\[
\lambda^n + p^*b_1 \lambda^{n-1} + \dots + p^*b_{n-1} \lambda + p^*b_n \in H^0(X,p^*K_X^{\otimes n}).
\]
Restricting $p$ to $S_b$ yields a finite morphism $p_b : S_b \to X$ of degree $n$.

One can further consider the pull-back of $p^*K_X^{\otimes n}$ to the product $\Tot(K_X) \times B$, which is naturally equipped with a section obtained by pull-back of $\lambda$. Seeing the $b_i$ as coordinates in $B$, one obtains a second section of our bundle, whose vanishing locus provides naturally a family of spectral curves $\Ss\subset \Tot(K_X) \times B$ for which we naturally have that $\Ss\cap(\Tot(K_X) \times \{b\}) = S_b$. Restricting the projection $p\times\mathbf{1}_B:\Tot(K_X) \times B \to X \times B$, we obtain a finite morphism of degree $n$:
\begin{equation} \label{eq def p}
p: \Ss \to X \times B. 
\end{equation}
For every $b \in B$, the corresponding spectral curve $S_b$ belongs to the linear system $|nX|$, and, by Bertini's theorem, it is generically smooth and irreducible. Furthermore, since the canonical divisor of the symplectic surface $\Tot(K_X)$ is zero, the genus of $S_b$ is given by
\begin{equation} \label{eq genus of the spectral curve}
d := g \left (S_b \right) = 1 + n^2(g-1).
\end{equation}
Thanks to Riemann--Roch theorem, $p_{*}\Oo_{S_b}$ is a degree
$-(n^2 - n)(g-1)$ vector bundle of rank $n$. This motivates the notation
\[
\delta := (n^2 - n)(g-1).
\]
Following \cite{BNR}, we consider the push-forward 
\begin{equation}\label{eq vector bundle}
E_L := p_{*}L 
\end{equation}
of a torsion free sheaf $L$ on $S_b$ of rank $1$ and degree $\delta$, which is a vector bundle on $X$ of rank $n$ and degree $\delta + \deg(\pi_*\Oo_{S_b}) = 0$. We consider as well the multiplication by the restriction to $S_b$ of tautological section, $\lambda_b : \Oo_{S_b} \to \Oo_{S_b} \otimes p^*K_X$. Note that this induces the following twisted endomorphism of $L$
\[%begin{equation} \label{eq tensoring by tautological section}
\morph{L}{L \otimes p^*K_X}{s}{s \otimes \lambda_b.}{}{\psi_b}
\]%end{equation}
whose push-forward returns the Higgs field 
\begin{equation} \label{eq Higgs field} 
\varphi = p_{*}\psi_b : E_L \to E_L \otimes K_X, 
\end{equation}
so that $b = h(\varphi)$. Thanks to the spectral correspondence \cite{hitchin_duke, BNR, simpson2, schaub, deCataldo}, each Hitchin fibre is identified with the compactified Jacobian of the corresponding spectral curve
\[
h^{-1} \left ( b \right ) \cong \overline{\Jac}^{\, \delta}(S_b).
\]

Fixing a point $x_0 \in X$ in our curve, we construct a smooth section $\hat{\sigma} : B \to \overline{\Jac}^{\, \delta}_{B}(S)$ by considering for every $b$ the line bundle $p_b^*\Oo_X(x_0)^{\otimes(n-1)(g-1)}$ on $S_b$. Such choice induces the following identification 
\begin{equation} \label{eq compactified Jacobian is trivial}
\overline{\Jac}^{\, \delta}_{B}(\Ss) \cong \overline{\Jac}^{\, 0}_{B}(\Ss).
\end{equation}

%--------------------------------------------------------------------------------------------------------------------------------------------------------------------------------------------------------------

\subsection{Fourier--Mukai transform and mirror symmetry for the Hitchin system}
\label{sc mirror}

A very active field of research is the occurrence of mirror symmetry phenomena between Higgs moduli spaces of pairs of Langlands dual groups. In the so-called semi-classical limit, mirror symmetry is expected to be realised via a Fourier--Mukai transform relative to the Hitchin fibration. In this section we review this correspondence within the framework of the moduli space of Higgs bundles for $\GL(n,\CC)$, which is Langlands self-dual. Hence, the mirror of $\M$ is (conjecturally) itself. Even if the Fourier--Mukai transform extends out of the locus of smooth Hitchin fibres \cite{arinkin, melo1, melo2}, we restrict here to the original construction of Mukai over (families of) abelian varieties. We do so because the locus of smooth Hitchin fibres is dense in the Higgs moduli space, hence the study of the duality there is enough for our purposes.

Let us denote by $B' \subset B$ the Zariski open subset given by those points $b \in B$ such that $S_b$ is smooth. We denote the restriction of $\Ss$ and $\M$ to $B'$ by 
\[
\Ss' := \Ss|_{B'}
\]
and 
\begin{equation} \label{eq M = Jac_B of S}
\M' = \M|_{B'} \cong \Jac^{\delta}_{B'}(\Ss').
\end{equation}
Note that all the points of $\M'$ are associated to line bundles over smooth spectral curves, which are automatically stable. Therefore, $\M'$ is contained in the stable locus,
\[
\M' \subset \M^\st.
\]
By the autoduality of smooth relative jacobians, we know that $\Jac_{B'}^0(\Ss') \cong \Jac_{B'}^{\delta}(\Ss')^\vee$; it thus follows from \eqref{eq M = Jac_B of S} that this is further isomorphic to $\Jac_{B'}^{\delta}(\Ss')^\vee \cong \M'$. Then, one can consider the commuting diagram 
\begin{equation}
  \label{sc duality of the Hitchin system on M'}
  \begin{gathered}
  \xymatrix{
    & \Jac^{\, \delta}_{B'}(\Ss') \times_{B'} \Jac^{\,\delta}_{B'}(\Ss')^\vee
    % \ar[d]^{\cong} & \\
    % & \Jac^{\, \delta}_{B'}(\Ss') \times_{B'} \Jac^{\, \delta}_{B'}(\Ss')^\vee 
    \ar[dd] \ar[rd]^{\check{\pi}} \ar[ld]_{\hat{\pi}} \\
    \M' \cong \Jac^{\, \delta}_{B'}(\Ss') \ar[rd]^{\hat{h}}
    & & \M' \cong \Jac^{\, \delta}_{B'}(\Ss')^\vee \ar[ld]_{\check{h}}\\
    & B', \ar@/^1pc/[ul]^{\hat{\sigma}} \ar@/_1pc/[ur]_{\check{\sigma}}
  }
  \end{gathered}
\end{equation}
where $\hat{\sigma}$ is the constant section considered in Section \ref{sc Hitchin fibration}, and $\check{\sigma}$ is the section given by considering the structural sheaf on each $\Jac^{\delta}(S_b)$.

We will study mirror symmetry in the sense of Strominger--Yau--Zaslow \cite{SYZ} in this context.

Since the relative scheme $\M' = \Jac^{\delta}_{B'}(\Ss')$ has a section $\hat{\sigma}$, it is well known (see \cite[8.2, Proposition 4]{bosh} for instance) that its relative Jacobian carries a Poincar\'e bundle 
\[
\Pp \to  \Jac_{B'}^{\delta}(\Ss') \times_{B'} \Jac_{B'}^{\delta}(\Ss')^\vee. 
\]
With it, we can consider the relative Fourier--Mukai transforms
\begin{equation} \label{eq def FM}
\morph{D^b_{B'}(\Jac^{\delta}_{B'}(\Ss')) \cong D^b_{B'}(\M')}{D^b_{B'}(\Jac^{\delta}_V(\Ss')^\vee) \cong D^b_{B'}(\M')}{F^\bullet}{\check{F}^\bullet := R\check{\pi}_* \left( \Pp \otimes \hat{\pi}^* F^\bullet \right )}{}{R\check{\FFf}}
\end{equation}
and
\begin{equation} \label{eq def FM*}
\morph{D^b_{B'}(\Jac^{\delta}_{B'}(\Ss')^\vee) \cong D^b_{B'}(\M')}{D^b_{B'}(\Jac^{\delta}_{B'}(\Ss')) \cong D^b_{B'}(\M')}{G^\bullet}{\hat{G}^\bullet := R\hat{\pi}_* \left( \Pp^* \otimes \check{\pi}^* G \right ).}{}{R\hat{\FFf}}
\end{equation}
After \cite{mukai}, this is an equivalence of categories since
\begin{equation} \label{eq inverse of FM}
R\hat{\FFf} \circ R\check{\FFf} = [d] \circ \left ( \Id_{\Jac}^{-1} \right )^*,
\end{equation}
where $\Id_{\Jac}^{-1}$ denotes the involution given by inverting elements on each $\Jac^{\delta}(S_b)$ under the group structure (recall \eqref{eq compactified Jacobian is trivial}), and $d$ is defined in \eqref{eq genus of the spectral curve}.

We say that a sheaf $F$ on $\Jac_{B'}^\delta(\Ss')$ is {\it $\ell$-WIT} (after {\it Weak Index Theorem}) if its image under $R\hat{\FFf}$ is a complex supported in degree $\ell$. Let us denote by $\WIT_\ell \left (\Jac_{B'}^\delta(\Ss') \right)$ the category of $\ell$-WIT sheaves. It follows from \eqref{eq inverse of FM} that the Fourier--Mukai transform $R\hat{\FFf}$ induces an equivalence of categories
\[
R\hat{\FFf} : \WIT_\ell\left (\Jac_{B'}^\delta(\Ss') \right) \stackrel{\cong}{\longrightarrow} \WIT_{d-\ell}\left (\Jac_{B'}^\delta(\Ss') \right)
\]
with inverse $R\check{\FFf}$.

\subsection{Flat unitary gerbes and the universal bundle}
\label{sc gerbes}

When the rank and the degree are coprime, the moduli space of stable Higgs bundles is fine and the universal bundle plays an important role in the construction of the Dirac--Higgs bundle (see Section \ref{sc Dirac-Higgs} below). In our case, with the degree being trivial, there is no universal bundle at hand, not even over the stable locus or any other Zariski open subset of the moduli space (this was proven for vector bundles by Ramanan \cite{ramanan} and reproved by Drezet and Narasimhan \cite{drezet&narasimhan} but a similar discussion holds for Higgs bundles). The best we have at hand is a local universal bundle in the \'etale topology constructed by Simpson \cite{simpson2}. 

Since our universal bundle only exists locally, we need to introduce gerbes. We refer to \cite{hitchin_lagrangian} for a nice introduction to gerbes. Given an algebraic variety $Y$, denote by $\Tors(\U(1), Y)$ the group of $\U(1)$-torsors over $Y$, {\it i.e.} the group of flat unitary line bundles on $Y$. A {\it flat unitary gerbe in the \'etale topology} on $Y$ is a sheaf of categories in the \'etale topology over $Y$ such that its restriction to each \'etale open subset $Y' \to Y$ is a torsor for the group $\Tors(\Un(1), Y')$. Given an \'etale covering $\{ Y_i \to Y \}_{i \in I}$, a gerbe provides a category (a groupoid indeed) for every $Y_i$, the natural transformations of these categories in the intersections $Y_{ij} := Y_i \times_Y Y_j$ are realized via tensoring by flat unitary line bundles $L_{ij} \to Y_{ij}$. Therefore, a gerbe defines a set of flat unitary line bundles over the intersections $\{ L_{ij} \to Y_{ij} \}_{i,j \in I}$ such that $L_{ij} \cong L_{ji}^{-1}$ and over the triple intersections $Y_{ijk} := Y_i \times_Y Y_j \times_Y Y_k$, one has that $\zeta_{ij}^* L_{ij} \otimes \zeta_{jk}^*L_{jk} \otimes \zeta_{ki}^*L_{ki}$ is isomorphic to the trivial bundle on $Y_{ijk}$, being $\zeta_{ij}$ the obvious projection $Y_{ijk} \to Y_{ij}$.

If $\beta$ is a flat unitary gerbe in the \'etale topology and $\{ Y_i \to Y \}_{i \in I}$ is an \'etale cover of $Y$, a {\it $\beta$-twisted sheaf} $\F$ is a set of sheaves $\{ F_i \to Y_i \}_{i \in I}$ such that on each intersection $\zeta_i^*F_i \cong L_{ij} \otimes \zeta_j^*F_j$, where $\zeta_i$ denotes the projection $Y_i \times_Y Y_j \to Y_i$. In this case, we also say that the cover $\{ Y_i \}_{i \in I}$ is {\it fine enough} for $\F$. A $\beta$-twisted sheaf $\F$ is a {\it $\beta$-twisted vector bundle} of rank $n$ if all the $F_i$ are locally free sheaves of rank $n$. We say that $\boldsymbol{\nabla} = \{ \nabla_i \}_{i \in I}$ is a connection on the $\beta$-twisted vector bundle $\F$ if $\nabla_i$ is a connection on each of the $F_i$ satisfying the compatibility relations $\zeta_i^* \nabla_i \cong \Id_{L_{ij}} \otimes \zeta_j^*\nabla_j + \nabla_{ij} \otimes \Id_{\zeta_j^*F_j}$, where $\nabla_{ij}$ is the flat unitary connection naturally defined on $L_{ij}$ via the flat unitary gerbe $\beta$. Given two $\beta$-twisted sheaves $\F_1 = \{ F_{1,i} \to Y_i \}_{i \in I}$ and $\F_2 = \{ F_{2,i} \to Y_i \}_{i \in I}$ that are fine enough for the same \'etale covering, a morphism $\boldsymbol{\psi} : \F_1 \to \F_2$ of $\beta$-twisted sheaves is a collection of morphisms of sheaves $\{ \psi_i : F_{1,i} \to F_{2,i} \}_{i \in I}$ satisfying $\zeta_i^*\psi_i \cong \Id_{L_{ij}} \otimes \zeta_j^*\psi_j$ for any $i,j \in I$. In a similar way, the notions of quotient of sheaves, complex of sheaves and cohomology generalize naturally to the context of $\beta$-twisted sheaves. If $f : Z \to Y$ is a morphism of algebraic varieties, $\beta$ a flat unitary gerbe in the \'etale topology over $Y$ and $\F$ a $\beta$-twisted sheaf on $Y$, we define $f^*\F$ to be $\{ f_i^* F_i \to Y_i \times_{Y} Z \}$, where $f_i : Y_i \times_{Y} Z \to Y_i$ is the projection to the first factor and we note that $\{ Y_i \times_{Y} Z \to Z \}_{i \in I}$ is an \'etale covering. Similarly, for a $f^*\beta$-twisted sheaf $\F$ on $Z$ with a fine enough \'etale covering of the form $\{ Y_i \times_{Y} Z \to Z \}_{i \in I}$, we set $f_* \F := \{ f_{i,*} F_i \to Y_i \}_{i \in I}$. Thanks to the projection formula and base-change theorems, $f_* \F$ is a $\beta$-twisted sheaf as
\begin{align*}
(\zeta_i \times_Y \Id_{Z})^* \left ( f_{i,*}F_i \right ) & \cong f_{ij,*} \left ( \zeta_i^* F_i \right )
\\
& \cong f_{ij,*} \left ( L_{ij} \otimes \zeta_j^*F_j \right )
\\
& \cong f_{ij,*} \left ( f_{ij}^*L_{ij} \otimes \zeta_j^*F_j \right ) 
\\
& \cong L_{ij} \otimes f_{ij,*}\zeta_j^*F_j
\\
& \cong L_{ij} \otimes (\zeta_j \times_Y \Id_{Z})^*f_{i,*}F_j,
\end{align*}
where $f_{ij}$ denotes the projection $Y_{ij} \times_{Y} Z \to Y_{ij}$. 

Given a $\beta$-twisted sheaf $\F_1$ and vector bundle $F_2$ over $Y$, we define $\F_1 \otimes F_2$ to be the $\beta$-twisted sheaf $\{ F_{1,i} \otimes y_i^*F_2  \to Y_i \}_{i \in I}$ where we denote by $y_i$ the \'etale maps $Y_i \to Y$.

As we said at the beginning of this section, Simpson \cite[Theorem 4.7 (4)]{simpson2} ensures the existence of local universal Higgs bundles $\left ( U_i , \Phi_i \right ) \to X \times Z_i$ for a certain \'etale covering $\{ Z_i \to \M^\st \}_{i \in I}$ of the stable locus. By universality, there exists a line bundle over each intersection $L_{ij} \to Z_i \times_Y Z_j$ satisfying $U_i \cong \pi_{ij}^* L_{ij} \otimes U_{j}$, where $\pi_{ij}$ is the obvious projection $X \times Z_i \times_Y Z_j \to Z_i \times_Y Z_j$. This defines a flat unitary gerbe in the \'etale topology over $\M^\st$ that we denote by $\beta$ for the rest of the paper. After having set our gerbe $\beta$, observe that $\U = \{ U_i \to X \times Z_i \}_{i \in I}$ is a $\pi_\M^*\beta$-twisted vector bundle over $\M^\st$, where $\pi_\M : X \times \M^\st \to \M^\st$ is the obvious projection. In view of this, we refer to
\[
\left ( \U, \boldsymbol{\Phi} \right ) = \left \{ \left ( U_i , \Phi_i \right ) \to X \times Z_i \right \}_{i \in I},
\]
as the {\it universal $\pi_\M^*\beta$-twisted Higgs bundle}. This object will be crucial in our description of the Dirac Higgs bundle which we shall address in the next subsection.

We finish this section adapting Kapustin--Witten's definition of a $\BBB$-brane \cite{kapustin&witten} to this context. Suppose that $Y$ is equipped with a hyperk\"ahler structure, we say that $(\F, \boldsymbol{\nabla})$ is a {\it space filling $\BBB$-brane} if $\F$ is a $\beta$-twisted vector bundle and $\boldsymbol{\nabla}$ is a connection which is hyperholomorphic, {\it i.e.} is $(1,1)$ with respect to the three complex structures of $Y$.

%%%%%%%%%%%%%%%%%%%%%%%%%%%%%%%%%%%%%%%%%%%%%%%%%%%%%%%%%%%%%%%%%%%%%%%%%%%%%%%%%%%%%%%%%%%%%%%%%%%%%%%%%%%%%%%%%%%%%%%%%%%%%%%%%%%%%%%%%%%%%%%%%%%%%%%%%%%%%%%%%%%%%%%%%%%%%%%%%%%%%%%%%%%%%%%%%%%%%%%%%%%%%%%

\subsection{The Dirac--Higgs bundle}

\label{sc Dirac-Higgs}

Hitchin (see \cite{hitchin_dirac} for instance) constructed the Dirac--Higgs bundle over the moduli space of Higgs bundle with coprime rank and degree, showing that it can be equipped with a hyperholomorphic structure. The Dirac--Higgs bundle was used by Hausel \cite{hausel} to study the cohomology of this moduli space. Blaavand \cite{blaavand} studied this object over the moduli space $\M^\st$, showing that it exists only as a differential object. In a local sense, Hitchin's proof of the existence of a  hyperholomorphic structure on the Dirac--Higgs bundle extends to this case, allowing us to produce $\BBB$-branes out of it.

Our ultimate goal is to study the mirror dual of $\BBB$-branes constructed out of the Dirac--Higgs bundle. For this task, we shall Fourier--Mukai transform the sheaves underlying these $\BBB$-branes. As the Fourier--Mukai transform is an algebraic device, we need an algebraic construction of the Dirac--Higgs bundle, which we address in this section. To surpass the non-existence of a universal bundle, we make use of the flat unitary gerbe $\beta$ defined in the \'etale topology of $\M^\st$.

Let us fix a Hermitian metric on the rank $n$ topologically trivial $C^\infty$-bundle $\EE$ over $X$. Associated to it, consider the vector space (of infinite dimension)
\[
\underline{\Omega} := \Omega_X^{1,0}(\EE)\oplus\Omega_X^{0,1}(\EE),
\]
which comes equipped with a natural metric. Given a Higgs bundle $\Ee = (E, \varphi)$ supported on $\EE$, we write $\overline{\partial}_E$ for the associated Dolbeault operator and $\partial_E$ for the $(1,0)$-part of the Chern connection constructed with the metric and $\overline{\partial}_E$. Hitchin introduced in \cite{hitchin_dirac} the following \emph{Dirac--Higgs operators}
\[
\DDd_{\Ee} = \begin{pmatrix} \partial_E & -\varphi \\  \varphi^* & -\overline{\partial}_E \end{pmatrix} \, :  \begin{pmatrix} \Omega_X^{0}(\EE) \\ \Omega_X^{0}(\EE) \end{pmatrix} \longrightarrow \underline{\Omega} = \begin{pmatrix} \Omega_X^{1,0}(\EE) \\ \Omega_X^{0,1}(\EE) \end{pmatrix}
\]
and
\[
\DDd_{\Ee}^* = \begin{pmatrix} \overline{\partial}_E &  \varphi \\ \varphi^* & \partial_E \end{pmatrix} :  \, \underline{\Omega}= \begin{pmatrix} \Omega_X^{1,0}(\EE) \\ \Omega_X^{0,1}(\EE) \end{pmatrix} \longrightarrow \begin{pmatrix} \Omega_X^{1,1}(\EE) \\ \Omega_X^{1,1}(\EE) \end{pmatrix}.
\]
Let $\HH^p(\Ee)$ denote the hypercohomology groups of the complex of sheaves $E\stackrel{\varphi}{\rightarrow}E\otimes K_X$. The key fact about such operators is that 
$$ 
\ker \DDd_{\Ee} \simeq \HH^0(\Ee)\oplus \HH^2(\Ee)$$
and
\begin{equation}\label{iso-h-coho}
\ker \DDd_{\Ee}^* \simeq \HH^1(\Ee),
\end{equation}
see \cite[Section 7]{hitchin_dirac} or \cite[Lemma 2.4.1 and Remark 2.4.2]{blaavand}. Hausel proved in \cite[Corollary 5.1.4]{hausel} that if $\Ee$ is a nontrivial stable Higgs bundle of degree 0, \emph{i.e.} $\Ee\ne(\mathcal{O}_X,0)$, then $\HH^0(\Ee)=\HH^2(\Ee)=0$, so that $\ker \DDd_{\Ee}=0$.
Since the index of $\DDd_{\Ee}$ is $-2 n (g - 1)$, see \cite[Lemma 2.1.8]{blaavand} or \cite[page 1226]{frejlich&jardim}, we conclude that
\begin{equation} \label{eq ker D* fixed}
\dim \ker \DDd_{\Ee}^* = 2 n (g - 1)
\end{equation}
whenever $\Ee$ is a semistable Higgs bundle of degree 0 of rank $n$ without trivial factors. Furthermore, the vector space $\ker \DDd_{\Ee}^*$ does not depend upon the choice of a representative within the S-equivalence class $[\Ee]$ of $\Ee$. 

Applying the Dirac operators pointwise in $\Bb^\st$ one obtains the corresponding morphism of trivial bundles
\[
\DDd^* : \ul{\Omega} \times \Bb^\st \longrightarrow \left ( \Omega_X^{1,1}(\EE) \oplus \Omega_X^{1,1}(\EE) \right ) \times \Bb^\st.
\]
After \eqref{eq ker D* fixed}, the kernel $\ker \DDd^*$ defines a $C^\infty$ vector bundle of rank $2n(g-1)$ over $\Bb^{\, \st}$. If ever $\ker \DDd^*$ descends from $\Bb^{\, \st}$ to give a vector bundle on $\M^{\, \st}$ we call the resulting object the {\it Dirac--Higgs bundle}. When our moduli space is equipped with a universal family, the Dirac--Higgs bundle is defined as the pull-back of $\Dd$ under the section $\M^{\, \st} \to \Bb^{\, \st}$ obtained from the universal family. The rank one case, where $\M_1^\st = \M_1$, is one of the few cases were the construction of the Dirac--Higgs bundle is possible, and it was achieved by the second named author in \cite{frejlich&jardim}.

For $n>1$ (and trivial degree) we have already seen that no universal bundle only exists, not even Zariski locally. As we have seen in Section \ref{sc gerbes}, for general rank, the best we can obtain is the $\pi_M^*\beta$-twisted universal bundle $(\U, \boldsymbol{\Phi}) = \{ (U_i, \Phi_i) \to X \times Z_i) \}$ over the moduli space of stable Higgs bundles $\M^\st$. We can now define the family of Dirac-type operators
\begin{equation} \label{eq pointwise definition of Dd*_U}
\DDd_{(U_i, \Phi_i)}^* : \underline{\Omega} \times Z_i \longrightarrow \left ( \Omega_X^{1,1}(\EE) \oplus \Omega_X^{1,1}(\EE) \right ) \times Z_i,
\end{equation}
given by $\left . \DDd_{(U_i, \Phi_i)}^* \right |_z = \DDd_{(U_i, \Phi_i)|_z}^*$. Recalling \eqref{eq ker D* fixed}, let us denote, for every $i \in I$, the rank $2 n (g - 1)$ algebraic sub-bundle of $\underline{\Omega} \times Z_i$
\[
D_i := \ker \DDd_{(U_i, \Phi_i)}^*, 
\]
and consider 
\[
\D := \left \{ D_i \to Z_i \right \}_{i \in I},
\]
which we shall call the \emph{$\beta$-twisted Dirac--Higgs bundle}.

Adapting the work of Hausel \cite{hausel}, one can describe the
$\beta$-twisted Dirac--Higgs bundle in terms of the universal bundle
by means of \eqref{iso-h-coho}. This will show that the
$\beta$-twisted Dirac--Higgs bundle is, indeed, a $\beta$-twisted
bundle over $\M^\st$, what justifies its name. 

\begin{proposition} \label{pr Hausel's description of Dirac-Higgs}
Consider the obvious projections $\pi_{\M} : X \times \M^\st \to
\M^\st$ and $\pi_{X} : X \times \M^\st \to X$. The $\beta$-twisted
Dirac--Higgs bundle is a $\beta$-twisted bundle over $\M^\st$
isomorphic to 
\begin{equation} \label{eq DD in terms of UU}
  \D \cong
  \mathbb{R}^1 \pi_{\M, *} \left(
    \U \stackrel{\boldsymbol{\Phi}}{\longrightarrow}
    \U \otimes \pi_X^* K_X
  \right). 
\end{equation}
\end{proposition}

\begin{proof}
Assuming the isomorphism \eqref{eq DD in terms of UU}, it follows from
the projection formula that $\D$ is a $\beta$-twisted bundle over
$\M^\st$. We then focus on the proof of this isomorphism. Considering
each local universal bundle $(U_i, \Phi_i)$ over the \'etale open open
subset $Z_i \to \M^\st$ and recall from Section \ref{sc gerbes} that
we denoted $\pi_{\M,i} : X \times Z_i \to Z_i$. Note that 
\[
  \RR^0 \pi_{\M, i, *} \left(
    U_i \stackrel{\Phi_i}{\longrightarrow} U_i \otimes \pi_X^* K_X
  \right)
  = \RR^2 \pi_{\M, i, *} \left(
    U_i \stackrel{\Phi_i}{\longrightarrow} U_i \otimes \pi_X^* K_X
  \right) = 0
\]
since, as observed above, $\HH^0(\Ee)=\HH^2(\Ee)=0$ for each $\Ee\in \M^\st$. It follows from \cite[Corollary 5.1.4]{hausel} that  
\[
D_i \cong \RR^1 \pi_{\M, i, *} \left (U_i \stackrel{\Phi_i}{\longrightarrow} U_i \otimes \pi_X^* K_X \right ).
\]
Globally, one gets \eqref{eq DD in terms of UU}, and the proof is completed.
\end{proof}

One can also define a connection on the $\beta$-twisted Dirac--Higgs
bundle $\D = \left \{ D_i \to Z_i \right \}_{i \in I}$. Consider the
trivial connection $d_i : \underline{\Omega} \times Z_i \to
\underline{\Omega}\otimes K_{Z_i}$, where $K_{Z_i}$ is the sheaf of
first order differentials in $Z_i$. Denote by $\Id_{K}$ the identity
in $K_{Z_i}$. Consider also the embedding $j_i : D_i \hookrightarrow
\underline{\Omega} \times Z_i$ and the projection $\pr_i :
\underline{\Omega} \times Z_i \to D_i$ defined by the natural metric
on $\underline{\Omega} \times Z_i$. Let us consider the connection
given by the composition 
\[
\nabla_i = (\pr_i \otimes \Id_K) \circ d_i \circ j_i
\]
and note that this defines a connection on the $\beta$-twisted
Dirac--Higgs bundle $\D$, 
\[
\boldsymbol{\nabla} = \{ \nabla_i : D_i \to D_i \otimes K_{Z_i}  \}_{i \in I}
\]
that we call the {\it Dirac--Higgs} connection. The importance of this
connection comes from fact that the Dirac--Higgs connection is of type
$(1,1)$ with respect to all complex structures, see in \cite[Theorem 2.6.3]{blaavand}.
%and in \cite[Proposition 11]{frejlich&jardim} for rank $1$ Higgs bundles.

One can check that
the $\nabla_i$ satisfy the compatibility conditions stated in Section
\ref{sc gerbes} and $\boldsymbol{\nabla}$ is a connection on the
$\beta$-twisted Dirac--Higgs bundle $\D$. Hence, it equips the
Dirac--Higgs bundle with a hyperholomorphic structure and $(\D,
\boldsymbol{\nabla})$ is a space-filling $\BBB$-brane on $\M^\st$.

%------------------------------------------------------------------
%------------------------------------------------------------------

\section{Tensorization and spectral data}
\label{sc tensorization}

The morphism between Higgs moduli spaces that one obtains by
considering the tensor product with a particular Higgs bundle will be
crucial in our description of the Nahm transform of high rank that we
provide in Section \ref{sc high rank Nahm} below. In this section we
explore the behaviour of the spectral data under tensorization,
generalizing partial results established in \cite{bradlow&schaposnik}
for Higgs bundles of rank $2$ and $4$. This is required in Section
\ref{sc Fourier-Mukai on Nahm transform}, during our study of the
mirror branes dual to the ones we obtain after Nahm transform of high
rank. 

As it is useful for the purpose of the remaining sections, we add to the notation related to the moduli space of Higgs bundles a sub-index indicating the rank.

Let us introduce in this section the tensorization of two Higgs bundles $\Ee = (E, \varphi)$ and $\Ff = (F, \phi)$, 
\[
\Ee \otimes \Ff := \left ( E \otimes F, \varphi \otimes \Id_F + \Id_E \otimes \phi \right ).
\]
It is well known that, if $\Ee$ and $\Ff$ are semistable, then $\Ee \otimes \Ff$ is semistable too. Then, fixing some $\Ee \in \M_n$, one can define a map
\begin{equation} \label{eq definition of tau}
\morph{\M_m}{\M_{nm}}{[\Ff]}{[\Ee \otimes \Ff].}{}{\tau^\Ee_m}
\end{equation}

\begin{remark} \label{rm tau hyperholomorphic}
Note that $\tau^\Ee_m$ is hyperholomorphic, meaning that it is a holomorphic morphism between $(\M_m, \Gamma_m^i)$ and $(\M_{nm}, \Gamma_{nm}^i)$ for each of the $i = 1$, $2$ or $3$. As it is defined, $\tau^\Ee_m$ is clearly holomorphic for $i = 1$. To see that it is also holomorphic for $i = 2$, consider the vector bundle with flat connection $(E', \nabla_{E'})$ corresponding to $\Ee$ under the Hitchin--Kobayashi correspondence, and observe that $\tau^\Ee_m$, in the complex structure $\Gamma_2$, sends the vector bundle with flat connection $(F', \nabla_{F'})$ corresponding to $\Ff$, to $(E' \otimes F', \nabla_{E'} \otimes \Id_{F'} + \Id_{E'} \otimes \nabla_{F'} )$. If $\tau^\Ee_m$ is holomorphic for $i = 1$ and $i = 2$, it is also holomorphic for $i = 3$ since this complex structure is given by the composition of the previous two.
\end{remark}

The map given by the sum along the fibres of the canonical line bundle,
\[
\sigma : K_X \times_X K_X \longrightarrow K_X,
\]
will be necessary for the description of $\tau^\Ee_m$ under the spectral correspondence, which we address next.

\begin{proposition} \label{pr description of Sigma_Phi}
Let $\Ee$ be a semistable Higgs bundle with spectral data $(S_\Ee,
L_\Ee)$ and $\Ff$ a (stable) Higgs bundle, associated to $(S_\Ff,
L_\Ff)$ with $S_\Ff$ smooth.
The tensor product
$\Ee \otimes \Ff
= (E \otimes F, \varphi \otimes \Id_F + \Id_E \otimes \phi)$
is a semistable Higgs bundle with spectral data
$(L_{\Ee \otimes \Ff},S_{\Ee \otimes \Ff})$ satisfying
\begin{equation} \label{eq description of Sigma_Phi}
S_{\Ee \otimes \Ff} = \sigma(S_{\Ee} \times_X S_{\Ff}),
\end{equation}
and
\[
L_{\Ee \otimes \Ff} \cong  \sigma_* \left ( q_\Ee^*L_\Ee \otimes q_\Ff^*L_\Ff \right ),
\]
where $q_\Ee$ and $q_\Ff$ denote the projections from $S_\Ee \times_X S_\Ff$ to $S_\Ee$ and $S_\Ff$.
\end{proposition}

\begin{proof}
Since $S_\Ff$ is smooth, $L_\Ff$ is a line bundle over $S_\Ff$ and $q_\Ff$ is a smooth morphism. This implies that $\Ff$ is stable, so the tensor product $\Ee \otimes \Ff$ is semistable as $\Ee$ is so. 

By construction, $S_{\Ee \otimes \Ff}$ is a projective curve contained in $\Tot(K_X)$, and the restriction of the projection morphism $p_{\Ee \otimes \Ff} : S_{\Ee \otimes \Ff} \to X$ is an $nm$-cover. Being defined as the push-forward under $\sigma$, if $L_{\Ee \otimes \Ff}$ has a sub-sheaf of dimension $0$, so does $q_\Ee^*L_\Ee \otimes q_\Ff^*L_\Ff$. We have that $q_\Ee^*L_\Ee$ is torsion free, as $q_\Ee$ is a smooth morphism and $L_\Ee$ is torsion free. Since $q_\Ff^*L_\Ff$ is a line bundle, $q_\Ee^*L_\Ee \otimes q_\Ff^*L_\Ff$ is also torsion-free, and this implies that $L_{\Ee \otimes \Ff}$ is torsion free as well, so the pair $(S_{\Ee \otimes \Ff}, L_{\Ee \otimes \Ff})$ is the spectral data of some Higgs bundle. The proof would be completed if we show that this Higgs bundle is indeed $\Ee \otimes \Ff$. In view of \eqref{eq vector bundle} and \eqref{eq Higgs field}, we need to show that $E \otimes F$ and $\varphi \otimes \Id_F + \Id_E \otimes \phi$ arise as the push-forward under $p_{\Ee \otimes \Ff}$ of $L_{\Ee \otimes \Ff}$ and the morphism $\psi_{\Ee \otimes \Ff} : L_{\Ee \otimes \Ff} \to L_{\Ee \otimes \Ff} \otimes p_{\Ee \otimes \Ff}^*K_X$ given by tensoring with the tautological section $\lambda_{\Ee \otimes \Ff}$ of $p_{\Ee \otimes \Ff}^* K_X$. 

Note that $p_{\Ee \otimes \Ff}$ fits in the commuting diagram 
\[
\xymatrix{
S_{\Ee} \times_X S_{\Ff} \ar[rd]_{\pi} \ar[rr]^{\sigma} & & S_{\Ee \otimes \Ff} \ar[ld]^{p_{\Ee \otimes \Ff}}
\\
 & X, & 
}
\]
while $\pi$ fits in 
\begin{equation} \label{eq commutative diagram}
  \begin{gathered}
    \xymatrix{
      & S_{\Ee} \times_X S_{\Ff} \ar[ld]_{q_\Ee}^{m}
      \ar[rd]^{q_\Ff}_{n}
      \ar[dd]^{\pi}_{nm} \\ 
      S_\Ee \ar[rd]_{p_\Ee}^{n}
      & & S_\Ff \ar[ld]^{p_\Ff}_{m} \\
      & X,
    }    
  \end{gathered}
\end{equation}
which is commutative and its exterior is Cartesian. 

Denote by $\lambda_\Ee$ the tautological section of $p_\Ee^*K_X$ over $S_\Ee$ and by $\lambda_\Ff$ the tautological section of $p_\Ff^*K_X$ over $S_\Ff$. In view of the previous commutative diagrams, one has that $\sigma^* \lambda_{\Ee \otimes \Ff}$, $q_\Ee^*\lambda_\Ee$ and $q_\Ff^*\lambda_\Ff$ are all sections of $\pi^*K_X$, and one has the equality
\[
\sigma^*\lambda_{\Ee \otimes \Ff} = q_\Ee^* \lambda_\Ee + q_\Ff^* \lambda_\Ff.
\]
Then, considering the morphism 
\[
\widetilde{\psi} : q_\Ee^*L_\Ee \otimes q_\Ff^*L_\Ff \longrightarrow q_\Ee^*L_\Ee \otimes q_\Ff^*L_\Ff \otimes \pi^* K_X
\]
given by tensorization under $\sigma^*\lambda_{\Ee \otimes \Ff}$, one has the decomposition
\[
\widetilde{\psi} = \widetilde{\psi}_\Ee + \widetilde{\psi}_\Ff,
\]
where we define
\[
\widetilde{\psi}_\Ee :  q_\Ee^*L_\Ee \otimes q_\Ff^*L_\Ff \longrightarrow q_\Ee^*L_\Ee \otimes q_\Ff^*L_\Ff \otimes \pi^* K_X
\]
to be the tensorization under $q_\Ee^*\lambda_{\Ee}$, and similarly for $\widetilde{\psi}_\Ff$. Denoting by
\[
\psi_\Ee : L_\Ee \longrightarrow L_\Ee \otimes p_\Ee^* K_X
\]
the morphism obtained by tensorizarion by $\lambda_\Ee$, one has that
\[
\widetilde{\psi}_\Ee = q_\Ee^* \psi_\Ee \otimes \Id_{q_\Ff^* L_\Ff}.
\]
One can define $\psi_\Ff$ analogously, obtaining
\begin{equation} \label{eq wtpsi_Ff}
\widetilde{\psi}_\Ff \cong \Id_{q_\Ee^* L_\Ee} \otimes q_\Ff^* \psi_\Ff.
\end{equation}

Since $\sigma \circ p_{\Ee \otimes \Ff} = \pi = q_\Ff \circ p_\Ff$, 
\begin{align*}
p_{\Ee \otimes \Ff\!\!,\, *} L_{\Ee \otimes \Ff} & \cong p_{\Ee \otimes \Ff\!\!,\, *} \sigma_* (q_\Ee^*L_\Ee \otimes q_\Ff^*L_\Ff)
\\
& \cong \pi_* (q_\Ee^*L_\Ee \otimes q_\Ff^*L_\Ff) 
\\
& \cong p_{\Ff\!\!,\, *} q_{\Ff,*} (q_\Ee^*L_\Ee \otimes q_\Ff^*L_\Ff).
\end{align*}
Note that we obtain a similar expression for the last line using $\pi = q_\Ee \circ p_\Ee$. Also, observe that 
\[
p_{\Ee \otimes \Ff\!\!,\, *} \psi_{\Ee \otimes \Ff} \cong \pi_* \widetilde{\psi} \cong p_{\Ee,*} q_{\Ee,*} \widetilde{\psi}_\Ee + p_{\Ff\!\!,\, *} q_{\Ff\!\!,\, *} \widetilde{\psi}_\Ff.
\]
From now on, we shall focus only on the study of $p_{\Ff\!\!,\, *} q_{\Ff\!\!,\, *} \widetilde{\psi}_\Ff$ as the description of $p_{\Ee,*} q_{\Ee,*} \widetilde{\psi}_\Ee$ is completely analogous.

As $L_\Ff$ is a line bundle, the projection formula further gives
\[
p_{\Ee \otimes \Ff\!\!,\, *} L_{\Ee \otimes \Ff} \cong p_{\Ff\!\!,\, *} ((q_{\Ff\!\!,\, *} q_\Ee^*L_\Ee) \otimes L_\Ff)
\]
and, after \eqref{eq wtpsi_Ff},
\[
p_{\Ff\!\!,\, *} q_{\Ff\!\!,\, *} \widetilde{\psi}_\Ff \cong p_{\Ff\!\!,\, *} \left ( \Id_{(q_{\Ff\!\!,\, *}q_\Ee^* L_\Ee)} \otimes  \psi_\Ff \right ).
\]

Since the exterior arrows of \eqref{eq commutative diagram} provide a Cartesian diagram, and $p_\Ff$ is flat, flat base change allow us to identify
\begin{align*}
p_{\Ee \otimes \Ff\!\!,\, *} L_{\Ee \otimes \Ff} & \cong p_{\Ff\!\!,\, *} ((p_\Ff^* p_{\Ee,*} L_\Ee) \otimes L_\Ff)
\\
& \cong p_{\Ff\!\!,\, *} ((p_\Ff^* E) \otimes L_\Ff),
\end{align*}
and 
\[
p_{\Ff\!\!,\, *} q_{\Ff\!\!,\, *} \widetilde{\psi}_\Ff \cong p_{\Ff\!\!,\, *} \left ( \Id_{(p_\Ee^* E)} \otimes  \psi_\Ff \right ) .
\]

As $p_\Ff^* E$ is locally free, applying again the projection formula we obtain the desired equality
\begin{align*}
p_{\Ee \otimes \Ff\!\!,\, *} L_{\Ee \otimes \Ff} & \cong E \otimes p_{\Ff\!\!,\, *} L_\Ff 
\\
& \cong E \otimes F,
\end{align*}
and 
\begin{align*}
p_{\Ff\!\!,\, *} q_{\Ff\!\!,\, *} \widetilde{\psi}_\Ff \cong & \Id_{E} \otimes p_{\Ff\!\!,\, *} \psi_\Ff 
\\
& \cong \Id_{E} \otimes \phi.
\end{align*}
Similarly, one obtains
\[
p_{\Ee,*} q_{\Ee,*} \widetilde{\psi}_\Ee \cong \varphi \otimes \Id_F,
\]
hence
\[
p_{\Ee \otimes \Ff,\!\!,\, *} \psi_{\Ee \otimes \Ff} \cong \varphi \otimes \Id_F + \Id_{E} \otimes \phi,
\]
and the proof is concluded.
\end{proof}

Consider the semistable Higgs bundle $\Ee = (E, \varphi)$ associated to the spectral data $(L_\Ee, S_\Ee)$, where $L_\Ee$ is a line bundle. After Proposition \ref{pr description of Sigma_Phi}, the morphism
\[
\morph{\Jac^{\,\delta_{m}}_{B'_{m}}(\Ss'_{m})}{\overline{\Jac}^{\delta_{nm}}_{B_{nm}}(\Ss_{nm})}{L_\Ff \to S_\Ff}{\sigma_*(q_\Ee^*L_\Ee \otimes q_\Ff^*L_\Ff) \to S_{\Ee \otimes \Ff},}{}{\hat{\tau}^\Ee_m}
\]
corresponds to $\tau^\Ee_m$.

\begin{corollary} \label{co description of hat tau}
For every semistable Higgs bundle $\Ee = (E, \varphi)$ of rank $n$, the diagram
\[
\xymatrix{
\Jac_{B'_m}^{\, \delta_m}(\Ss'_m) \ar[d]_{p_{m,*}}^{\cong} \ar[rr]^{\hat{\tau}^\Ee_m} & & \overline{\Jac}_{B_{nm}}^{\delta_{nm}}(\Ss_{nm}) \ar[d]^{p_{nm, *}}_{\cong} 
\\
\M'_m \ar[rr]^{\tau^\Ee_m} & & \M_{nm},
}
\]
commutes.
\end{corollary}

Finally, we study the relation of $\tau^\Ee_m$ and the holomorphic $2$-forms $\Lambda^1_{m}$ and $\Lambda_{nm}^1$.

\begin{lemma} \label{lm tau and Lambda}
One has that
\[
\tau^{\Ee, *}_m \Lambda^1_{nm} (\cdot, \cdot) \cong n \Lambda^1_m  (\cdot, \cdot). 
\]
\end{lemma}

\begin{proof}
Recall that $g_{nm}$ is obtained from the Gauge invariant metric \eqref{eq HB metric}. Observe that the isomorphism $\EE_{nm} \cong \EE_n \otimes \EE_m$ induces the isomorphism of vector bundles 
\begin{equation} \label{eq Omega_nm = Omega_n x Omega_m}
\underline{\Omega}_{nm} \cong \left(\Omega_X^{1,0}(\EE_n \otimes \EE_m)\oplus\Omega_X^{0,1}(\EE_n \otimes \EE_m)\right),
\end{equation}
and note that the infinitesimal deformations of the Higgs bundles contained in the image of $\tau^\Ee_m$ are of the form $\Id_E \otimes \dot{\alpha}_m \in \Omega_X^{0,1}(\End(\EE_{nm}))$ and $\Id_E \otimes \dot{\varphi}_m \in \Omega_X^{1,0}(\End(\EE_{nm}))$. Then, one can easily check that $\tau^{\Ee, *}_m g_{nm} = \rk(\Ee) g_m$.

Since $\tau^\Ee_m$ commutes with all the complex structures $\Gamma^1$, $\Gamma^2$ and $\Gamma^3$, we have 
\begin{align*}
\tau^{\Ee, *}_m \Lambda^1_{nm} (\cdot, \cdot) & = \tau^{\Ee, *}_m \omega^2_{nm}(\cdot, \cdot) + \ii \tau^{\Ee, *}_m \omega^3_{nm}(\cdot, \cdot)
\\
& = \tau^{\Ee, *}_m g_{nm}(\cdot, \Gamma^2(\cdot)) + \ii \tau^{\Ee, *}_m g_{nm}(\cdot, \Gamma^3(\cdot))
\\
& = (\tau^{\Ee, *}_m g_{nm})(\cdot, \Gamma^2(\cdot)) + \ii (\tau^{\Ee, *}_m g_{nm})(\cdot, \Gamma^3(\cdot))
\\
& = n g_{m}(\cdot, \Gamma^2(\cdot)) + \ii n g_{m}(\cdot, \Gamma^3(\cdot))
\\
& = n \omega^2_{m}(\cdot, \cdot) + \ii n \omega^3_{m}(\cdot, \cdot)
\\
& = n \Lambda^1_{m}(\cdot, \cdot).
\end{align*}
\end{proof}

%------------------------------------------------------------------
%------------------------------------------------------------------

\section{Nahm transform of high rank}
\label{sc high rank Nahm}

In \cite{jardim_survey}, the Nahm transform is constructed for any a family of self-dual connections over a spin four-manifold with non-negative scalar curvature. The construction produces a vector bundle with connection over the parametrizing space of the family once we consider, at each point of the parametrizing space, the cokernel of the associated Dirac operator. It is also possible to define naturally a connection on this bundle. The Nahm transform for Higgs bundles is defined in \cite{frejlich&jardim} considering, for each stable Higgs bundle, the family obtained by twisting with the universal family of rank $1$ Higgs bundles. Hence, for each stable Higgs bundle, we obtain a Hermitian connection over $\M_1$ of type $(1,1)$ with respect to the complex structures $\Gamma^1_1$, $\Gamma^2_1$ and $\Gamma^3_1$. 

Here, we generalize the previous construction to moduli spaces of stable Higgs bundles of arbitrary rank. The main difference with the rank $1$ case relies in the fact that there is no universal bundle over $X \times \M_n^\st$ for $n > 1$, not even Zariski locally. Therefore, as we did in Section \ref{sc Dirac-Higgs}, we shall work with the gerbe $\beta_n$ in the \'etale topology of $\M_n$, and make use of the definition of $\beta_n$-twisted bundles. 

Fix a Higgs bundle $\Ee = (E, \varphi)$ of rank $n$ and degree $0$, supported on the Hermitian $C^\infty$ vector bundle $\EE_n$. For every rank $m$ Higgs bundle $\Ff = (F, \phi)$ with $\deg(F) = 0$, supported on the Hermitian $C^\infty$ vector bundle $\EE_m$, we can consider the Higgs bundle $\Ee \otimes \Ff = (E \otimes F, \varphi \otimes \Id_F + \Id_F \otimes \phi)$ on $\EE_n \otimes \EE_m \cong \EE_{nm}$. If $\Ee$ and $\Ff$ are semistable, it is well known that $\Ee \otimes \Ff$ is semistable, although, for $m > 1$, such correspondence is no longer valid when we replace semistability with stability. 

We recall that a crucial step in the construction of the Dirac--Higgs bundle is Hausel's vanishing statement \cite[Corollary 5.1.4]{hausel} which ensures that $\ker \DDd_{\Ee}=0$ whenever $\Ee$ is stable and different from the trivial Higgs bundle of degree 0 (where by trivial Higgs bundle we mean the pair $(\Oo_X,0)$ given by the trivial line bundle and zero Higgs field). Note that $\HH^0(\mathcal{O}_X,0)=\HH^2(\mathcal{O}_X,0)=\CC$. Since we intend to study the Dirac operator over the locus of Higgs bundles obtained as a tensor product, we need to study first whether or not one can generalize Hausel's vanishing statement to this locus. This justifies the following definition.

If $\Ee$ be a semistable Higgs bundle of degree 0, let ${\rm
  gr}(\Ee)=\oplus_j \Gg_j$ be the associated graded object, where
$(\Gg_1,\dots,\Gg_l)$ are the stable factors of its Jordan--H\"older
filtration; we say that $\Ee$ is \emph{without trivial factors}
%(or w.t.f., for short)
if none of these factors $\Gg_j$ is the trivial
Higgs bundle $(\mathcal{O}_X,0)$. Clearly, every nontrivial stable
Higgs bundle is without trivial factors.

\begin{lemma}\label{sst ker D}
Let $\Ee$ be a semistable Higgs bundle of degree 0. If $\Ee$ is
without trivial factors, then
\[
  \HH^0(\Ee)=\HH^2(\Ee)=0
  \qquad\text{and}\qquad
  \ker \DDd_{\Ee}^* = \HH^1(\gr(\Ee))
  \simeq \oplus_j \HH^1(\Gg_j),
\]
where $\Gg_j$ are the stable factors of the Jordan--H\"older
filtration of $\Ee$.  In addition, $\Ee$ is without trivial
factors if and only if $\ker \DDd_{\Ee}=0$.
\end{lemma}

%{\color{red}
%\noindent\textbf{Lemma} (old version).
%Let $\Ee$ be a semistable Higgs bundle of degree 0. If $\Ee$ is
%without trivial factors, then $\HH^0(\Ee)=\HH^2(\Ee)=0$.  Then 
%$\ker \DDd_{\Ee}=0$ if and only if $\Ee$ is without trivial
%factors. In addition, if $\Ee$ is without trivial factors, then
%\[
%  \ker \DDd_{\Ee}^* = \HH^1(\gr(\Ee))
%  \simeq \oplus_j \HH^1(\Ee_j),
%\]
%where $\Ee_j$ are the stable factors of the Jordan--H\"older
%filtration of $\Ee$. 
%}

\begin{proof}
Let $\Ee$ is a semistable Higgs bundle of degree 0 and let
$$ 0=\Ee_0 \subset \Ee_1\subset \cdots \subset \Ee_l=\Ee $$
be its Jordan--H\"older filtration, and let $\Gg_j=\Ee_j/\Ee_{j-1}$ be
its factors. We prove the first claim by induction on the length $l$
of the Jordan--H\"older filtration.

If $l=1$, then $\Ee$ is stable and the claim is just the corollary due to Hausel mentioned above. For the induction step, assume that the claim holds when the $l=k-1$ and consider the short exact sequence of Higgs bundles
$$ 0 \to \Ee_{l-1} \to \Ee_l=\Ee \to \Gg_l \to 0. $$
If $\Ee$ is without trivial factors, then $\Gg_l$ is a nontrivial stable Higgs bundle, and $\Ee_{l-1}$ is also without trivial factors, so that, by the induction hypothesis
$$ \HH^0(\Ee_{l-1})=\HH^2(\Ee_{l-1})=\HH^0(\Gg_l)=\HH^2(\Gg_l)=0. $$
The associated long exact sequence in hypercohomology then yields
$$ \HH^0(\Ee)=\HH^2(\Ee)=0  ~~ {\rm and} ~~  \HH^1(\Ee)=\HH^1(\Ee_{l-1})\oplus\HH^1(\Gg_l)= \bigoplus_{j=1}^l \HH^1(\mathcal{G}_j). $$
In particular, we also conclude that if $\Ee$ is without trivial factors, then $\ker\DDd_{\Ee}=0$.

For the converse statement, assume that one of the factors of the Jordan--H\"older filtration, say $\Gg_j$ with $j\in\{1,\dots,l\}$, is trivial. Considering the exact sequence
$$ 0 \to \Ee_{j-1} \to \Ee_j \to \Gg_j \to 0, $$
we obtain a surjective map $\HH^2(\Ee_j)\twoheadrightarrow \HH^2(\Gg_j)=\CC$, thus $\ker\DDd_{\Ee_j}\ne0$. The monomorphism $\Ee_j\hookrightarrow \Ee$ then provides an injective map $\ker\DDd_{\Ee_j}\hookrightarrow\ker\DDd_{\Ee}$, proving that $\ker\DDd_{\Ee}\ne0$ as well.
\end{proof}

Lemma \ref{sst ker D} shows that the rank of $\ker \DDd^*_\Ee$ does not jump if we remain inside the locus of Higgs bundle without trivial factors, but it will as long as we leave this locus. Below we find conditions under which $\Ee \otimes \Ff$ is without trivial factors.

\begin{lemma} \label{lm contained in Mwtf}
Let $\Ee = (E, \varphi)$ and $\Ff = (F, \phi)$ be semistable Higgs bundles of degree $0$ and rank $n$ and $m$, respectively.
\begin{enumerate}
\item if $n>m$ and $\Ee$ is stable, then $\Ee \otimes \Ff$ is without trivial factors;
\item if $n<m$ and $\Ff$ is stable, then $\Ee \otimes \Ff$ is without trivial factors;
\item if $n=m$, $\Ee$ and $\Ff$ are stable and $\Ff \ncong \Ee^*$, then $\Ee \otimes \Ff$ is without trivial factors.
\end{enumerate}
\end{lemma}

In particular, if $n\ne m$ and both $\Ee$ and $\Ff$ are stable, then $\Ee \otimes \Ff$ is without trivial factors.

\begin{proof}
The (semi)stability of $\Ee$ implies the (semi)stability of $\Ee^* = (E^*, -\varphi^t)$. If $\Ee \otimes \Ff$ has a trivial factor, then there exists $\psi : \Oo_X \to E \otimes F$ such that $(\varphi \otimes \Id_F + \Id_E \otimes \phi)(\psi) = 0$. Equivalently, there exists a nontrivial morphism $\psi : E^* \to F$ such that $(\psi \otimes \Id_{K_X}) \circ (-\varphi^t) = \phi \circ \psi$. As a consequence, the image $\im \psi$ is a $\phi$-invariant sub-sheaf of $F$ and its saturation $\overline{\im \psi}$ a $\phi$-invariant sub-bundle of $F$. Also, the kernel $\ker \psi$ is a $(-\varphi^t)$-invariant bundle of $E$. As $\deg(\overline{\im \psi}) > \deg(\im \psi)$, note that $\deg(\im \psi) = \deg(\ker \psi) = 0$ and $\overline{\im \psi} = \im \psi$ due to the semistability of $\Ee^*$ and $\Ff$ and the fact that both have trivial degree. Hence $\im \psi$ is a vector sub-bundle of $F$. If $n \neq m$, either $\im \psi$ or $\ker \psi$ are proper sub-bundles, contradicting the stability of $\Ff$ or $\Ee^*$, respectively. Finally, suppose that $n = m$, both $\Ee$ and $\Ff$ are stable and $\Ff$ is not isomorphic to the dual $\Ee^*$. Then $\Ee \otimes \Ff$ is without trivial factor, as otherwise the last condition will be violated.
\end{proof}

Denote by $\M_m^{\Ee}$ the open subset of $\M_m^\st$ given by those Higgs bundles $\Ff$ such that $\Ee \otimes \Ff$ is without trivial factors. In virtue of Lemma \ref{lm contained in Mwtf}, we require that $\Ee$ is stable when $\rk(\Ee) \geq m$, and under these conditions  
\begin{itemize}
\item when $\rk(\Ee) \neq m$, one has
    \[
    \M^\Ee_m = \M_m^\st;
    \]
\item while for $\rk(\Ee) = m$, we have
    \[
    \M_m^{\Ee} = \M_m^\st \setminus \{ \Ee^* \}.
    \]

\end{itemize}

Given an integer $m$ and a semistable $\Ee = (E, \varphi)$ ($\Ee$ stable if  $\rk(\Ee) \geq m$), we define, for each $\Ff = (F, \phi)$ in $\M_m^\Ee$, the following Dirac-type operators where we make use of the isomorphism \eqref{eq Omega_nm = Omega_n x Omega_m},
\[
\DDd^\Ee_\Ff := \DDd_{\Ee \otimes \Ff} : \Omega_X^{0}(\EE_n \otimes \EE_m)\oplus\Omega_X^{0}(\EE_n \otimes \EE_m) \longrightarrow \underline{\Omega}_{nm},
\]
and its adjoint
\[
\DDd^{\Ee,*}_\Ff := \DDd^*_{\Ee \otimes \Ff} : \underline{\Omega}_{nm} \longrightarrow \Omega_X^{1,1}(\EE_n \otimes \EE_m)\oplus\Omega_X^{1,1}(\EE_n \otimes \EE_m).
\]
After Lemma \ref{lm contained in Mwtf} and \eqref{eq ker D* fixed}, one has that
\[
\dim \ker \DDd^{\Ee,*}_\Ff = \dim \ker \DDd_{\Ee \otimes \Ff}^*  = 2 n m (g - 1)
\]
is fixed. As we did in the definition of the $\beta_n$-twisted Dirac--Higgs bundle in Section \ref{sc Dirac-Higgs}, take an \'etale covering $\{ Z_{m,i} \to \M_m^\Ee \}_{i \in I}$ which is fine enough for the gerbe $\beta_m$ and recall the $\beta_m$-twisted universal Higgs bundle $(\U_m, \boldsymbol{\Phi}_m) = \{ (U_{m,i}, \Phi_{m,i}) \to X \times Z_{m,i} \}_{i \in I}$ on $X \times \M^\st_m$. For all $i \in I$, consider the families of Dirac-type operators 
\[
\DDd^{\Ee,*}_{(U_{m,i}, \Phi_{m,i})} = \DDd_{\Ee \otimes (U_{m,i}, \Phi_{m,i})} : \underline{\Omega}_{nm} \times Z_{m,i} \longrightarrow \Omega_X^{1,1}(\EE_n \otimes \EE_m)\oplus\Omega_X^{1,1}(\EE_n \otimes \EE_m) \times Z_{m,i},
\]
defined point-wise as we did in \eqref{eq pointwise definition of Dd*_U}, and set 
\[
\widehat{E}_{m,i} := \ker \DDd^{\Ee,*}_{(U_{m,i}, \Phi_{m,i})},
\]
which is a holomorphic rank $2 n m(g - 1)$ bundle over each \'etale open subset $Z_{m,i}  \to \M^\Ee_m$. We consider the \'etale bundle 
\[
\widehat{\E}_m := \left \{ \widehat{E}_{m,i} \to Z_{m,i} \right \}_{i \in I}.
\]

Next, we construct a connection on $\widehat{\E}_m$. Recall that there is a natural metric on $\underline{\Omega}_{nm}$ and consider the induced projection $\pr_{m,i} : \underline{\Omega}_{nm} \times Z_{m,i} \to \widehat{E}_{m,i}$ defined by it. Observe that one naturally has a trivial connection $\underline{d}_{m,i} : \underline{\Omega}_{nm} \times Z_{m,i} \to \underline{\Omega}_{nm} \times Z_{m,i} \otimes K_{m,i}$ on the trivial bundle $\underline{\Omega}_{nm} \times Z_{m,i}$, and consider the embedding $j_{m,i} : \widehat{E}_{m,i} \hookrightarrow \underline{\Omega}_{nm} \times Z_{m,i}$. Define \[
\widehat{\nabla}_{m}^{\Ee} := \pr_{m,i} \circ \underline{d}_{m,i} \circ j_{m,i}
\]
giving the connection on $\widehat{\E}_m$
\[
\widehat{\boldsymbol{\nabla}}_{m}^{\Ee} := \{\widehat{\nabla}_{m}^{\Ee} : E_{m,i} \to E_{m,i} \otimes K_{Z_{m,i}} \}_{i \in I}.
\]
We define the {\it rank $m$ Nahm transform} of the Higgs bundle $\Ee$ as the pair
\begin{equation} \label{eq Nahm transform}
\widehat{\boldsymbol{\Ee}}_m := \left  ( \widehat{\E}_m, \widehat{\boldsymbol{\nabla}}_{m}^{\Ee}\right ). 
\end{equation}
It can be shown that the rank $m$ Nahm transform is a $\beta_m$-twisted bundle with connection by providing an analogous result to Proposition \ref{pr Hausel's description of Dirac-Higgs} obtained by adapting the work of Hausel \cite{hausel}.

\begin{proposition} \label{pr Nahm transform a la Hausel}
Consider the obvious projections $\pi_{X} : X \times \M^\st_m \to X$
and  $\pi_{\M} : X \times \M^\st_m \to \M^\st_m$. Then,
$\widehat{\E}_m$ is a $\beta_m$-twisted bundle over $\M^\st_m$
isomorphic to 
\begin{equation} \label{eq EE in terms of UU}
  \widehat{\E}_m \cong \mathbb{R}^1 \pi_{\M, *}
  \left(
    \pi_X^*E \otimes \U_m
    \stackrel{\pi_X^*\varphi\, \otimes \Id_{\U} + \Id_{\pi_X^*E}
      \otimes\,\mathbf{\Phi}_m
    }{\,\tikz[baseline=-.5ex]{\draw[->](0, 0) -- (3.5, 0);}\,}
    \pi_X^*E \otimes \U_m \otimes \pi_X^* K_X
  \right).
\end{equation}
\end{proposition}

\begin{proof}
This follows immediately from the proof of Proposition \ref{pr Hausel's description of Dirac-Higgs}.
\end{proof}

\begin{remark}
The Nahm transform defined by Frejlich and the second named author in \cite{frejlich&jardim} is precisely the case $m=1$ of the construction above. Note that, in this case, the gerbe $\beta_1$ is trivial, so $\widehat{\E}_1$ can be defined globally as a vector bundle over $\M_1$.
\end{remark}

\begin{remark}
Observe that, after Propositions \ref{pr Hausel's description of Dirac-Higgs} and \ref{pr Nahm transform a la Hausel}, the rank $m$ Nahm transform of the trivial Higgs bundle $\Oo := (\Oo_X, 0)$ coincides with $(\D_m, \boldsymbol{\nabla}_m)$, the rank $m$ $\beta_m$-twisted Dirac--Higgs bundle and connection,
\[
\hat{\Oo}_m \cong \left ( \D_m, \boldsymbol{\nabla}_m \right ).
\]
\end{remark}

Recall the morphism $\tau^\Ee_m$ defined in \eqref{eq definition of tau}. By definition of $\M^\Ee_m$, one has that every Higgs bundle in its image $\tau^\Ee_m(\M^\Ee_m) \subset \M_{nm}$ is without trivial factors although not necessarily stable. Over the open subset $(\tau^\Ee_m)^{-1}(\M^\st_{nm})$ of $\M^\Ee_m$ it is possible to give an alternative description of the rank $m$ Nahm transform, compare with \cite[Definition 3.0.2]{blaavand}.

\begin{proposition} \label{pr alternative definition of Nahm transform}
Consider the rank $m$ Nahm transform of a stable Higgs bundle $\Ee$ of rank $n$. One has that
\begin{equation} \label{eq definition of hat E}
\left . \widehat{\E}_m \right |_{(\tau^\Ee_m)^{-1}(\M^\st_{nm})}  \cong \tau^{\Ee,*}_m \D_{nm}
\end{equation}
and 
\begin{equation} \label{eq definition of hat nabla}
\left . \widehat{\boldsymbol{\nabla}}_{m}^{\Ee}\right |_{(\tau^\Ee_m)^{-1}(\M^\st_{nm})} \cong \tau^{\Ee,*}_m \boldsymbol{\nabla}_{nm}. 
\end{equation}
\end{proposition}

\begin{proof}
By the universal property of $(\U_{nm}, \boldsymbol{\Phi}_{nm})$, one has that the restriction to $X \times (\tau^\Ee_m)^{-1}(\M^\st_{nm})$ of the gerbe $\beta_{nm}$ is isomorphic to (the pull-back of) $\beta_m$, hence $\pi_X^*\Ee \otimes (\U_{m}, \boldsymbol{\Phi}_{m})$ is isomorphic to the restriction of $(\U_{nm}, \boldsymbol{\Phi}_{nm})$. The proof follows from this fact and Propositions \ref{pr Hausel's description of Dirac-Higgs} and \ref{pr Nahm transform a la Hausel}.
\end{proof}

We can then construct a new class of $\beta_m$-twisted $\BBB$-branes.

\begin{proposition-definition}\label{prop-def}
Given a semistable Higgs bundle $\Ee$ $(\Ee$ stable if $\rk(\Ee) \geq m)$, its rank $m$ Nahm transform $\widehat{\boldsymbol{\Ee}}_m$ is a $\beta_m$-twisted hyperholomorphic vector bundle over $\M_m^\Ee$ with a hyperholomorphic connection. In other words, $\widehat{\boldsymbol{\Ee}}_m$ is a space filling, $\beta_m$-twisted $\BBB$-brane over $\M_m^\Ee$ and we refer to it as the Nahm brane on $\M_m$ associated to $\Ee$.
\end{proposition-definition}

\begin{proof}
We have seen in Remark \ref{rm tau hyperholomorphic} that $\tau^\Ee_m$ is a hyperholomorphic morphism. Since the Dirac--Higgs connection is hyperholomorphic \cite[Theorem 2.6.3]{blaavand}, %and \cite[Proposition 11]{frejlich&jardim}),
thanks to Proposition \ref{pr alternative definition of Nahm transform}, one has that the higher rank Nahm transform is hyperholomorphic as well.
\end{proof}

\begin{remark} \label{rm other BBB}
One can also construct a $\beta_m$-twisted $\BBB$-brane on $\M_{nm}$ by considering the push-forward under $\tau^{\Ee}_{m}$ of the Dirac--Higgs bundle and connection $(\D_{m}, \boldsymbol{\nabla}_{m})$ on $\M_{m}^{\st}$.
\end{remark}

\begin{remark}
Note that the fibre of $\widehat{\E}_m$ over the point of the moduli space $\Ff \in \M_m^\st$ can be identified with the first hypercohomology space $\HH^1(\Ee \otimes \Ff)$, classifying the extensions of $\Ee^*$ by $\Ff$ (within the category of Higgs bundles). 
\end{remark}

%-------------------------------------------------------------------
%-------------------------------------------------------------------

\section{Mirror branes}

\subsection{A Fourier--Mukai transform for bundles twisted by gerbes}
\label{sc FM for twisted vector bundles}

The goal of this section is to describe the mirror partners of the Nahm $\BBB$-branes we have constructed in previous sections. We shall do that by Fourier--Mukai transforming the Dirac--Higgs bundle and the bundles we obtain from the high rank Nahm transform. Before that, we first need to adapt the Fourier--Mukai transform to the setting of vector bundles twisted with a gerbe, a task that we address in this subsection. This requires the introduction of some notation.

Let $Y \to V$ be a smooth $V$-variety equipped with a section $\hat{\sigma} : V \to Y$. In that case (see \cite[8.2, Proposition 4]{bosh}), there exists a relative Jacobian $\Jac^0_V(Y)$ and let us consider $Y^\vee$ to be a torsor for $\Jac^0_V(Y)$. Consider an \'etale covering $\{ V_i \to V \}_{i \in I}$ of $V$ and observe that $\{ \zeta_i: Y \times_V V_i \to Y \}_{i \in I}$ and $\{ \zeta_i^\vee : Y^\vee \times_V V_i \to Y^\vee \}_{i \in I}$ are \'etale coverings of, respectively, $Y$ and $Y^\vee$. Let $\beta$ be a flat unitary gerbe on the \'etale topology over $Y$ giving the set of flat unitary line bundles $\{ L_{ij} \to Y \times_V V_{ij} \}_{i,j \in I}$ over the intersections $(Y \times_V V_i) \times_V (Y \times_V V_j) \cong Y \times (V_i \times_V V_j) = Y \times_V V_{ij}$. These flat line bundles define naturally sections $\check{\sigma}_{ij} : V_{ij} \to \Jac^0_{V_{ij}}(Y)$ which  can be understood as $V_{ij}$-automorphisms of $Y^\vee$. By abuse of notation, we still denote them by 
\[
\check{\sigma}_{ij} \, : \, Y^\vee \times_V V_{ij} \longrightarrow Y^\vee \times_V V_{ij}.
\]

In this context, if $Z \to V$ is another $V$-variety, we say that $\fF : Z \to Y^\vee$ is a {\it $\beta$-shifted morphism of $V$-varieties} if it is a collection of morphisms $\{ f_i : Z \times_V V_i \to Y^\vee \times_V V_i \}_{i \in I}$ such that
\[
\left . (f_i \times_V \Id_{V_j}) \right |_{Z \times_V V_{ij}} = \left . \check{\sigma}_{ij} \circ (f_j \times_V \Id_{V_i}) \right |_{Z \times_V V_{ij}},
\]
for every $i,j \in I$. When all the $f_i$ satisfy a certain property of morphisms of varieties, we say that the $\beta$-shifted morphism $\fF$ has this property. For instance, if $Y^\vee$ is equipped with a symplectic form $\Lambda$, we say that the image of $\fF : Z \to Y^\vee$ is {\it Lagrangian}, if $f_i(Z \times_V V_i)$ is a Lagrangian subvariety of $Y^\vee \times_V V_i$ with respect to pull-back of the symplectic form $(\zeta_i^\vee)^*\Lambda$.

Also, we say that a {\it $\beta$-shifted coherent sheaf} $\Ggg$ is a collection of coherent sheaves $\{ G_i \in \Coh \left (Y^\vee \times_V V_i \right ) \}_{i \in I}$ such that for every $i,j \in I$ one has
\[
(\Id_{Y^\vee} \times_V \zeta^\vee_i)^* G_i = \check{\sigma}_{ij, *} (\Id_{Y^\vee} \times_V \zeta^\vee_j)^* G_j
\]
as sheaves over $Y^\vee \times_V V_{ij}$. If $\fF : Z \to Y^\vee$ is a $\beta$-shifted closed immersion of $V$-varieties, and $G$ is a coherent sheaf over $Z$, we define {\it $\beta$-shifted push-forward} $\fF_*G$ as the collection $\{ f_{i,*} z_i^*G \to Y^\vee \times_V V_i \}_{i \in I}$, where $z_i$ denotes the \'etale morphism $Z \times_V V_i \to Z$. Observe that $\fF_*G$ is naturally a $\beta$-shifted coherent sheaf as for every $i,j \in I$, one has
\begin{align*}
(\Id_{Y^\vee} \times_V \zeta_i)^* f_{i,*} z_i^*G & \cong (f_i \times_V 1_{V_j})_*  (\Id_{Z} \times_V \zeta_i)^* z_i^*G
\\
& \cong (f_i \times_V 1_{V_j})_* z_{ij}^*G
\\
& \cong \check{\sigma}_{ij,*} (f_j \times_V 1_{V_i})_* z_{ij}^*G
\\
& \cong \check{\sigma}_{ij,*} (f_j \times_V 1_{V_i})_* (\Id_{Z} \times_V \zeta_j)^* z_j^*G
\\
& \cong \check{\sigma}_{ij,*} (\Id_{Y^\vee} \times_V p_j)^* f_{j,*} z_j^*G
\end{align*}
thanks to the base-change theorems and the definition of $\beta$-shifted morphism.

Let us now adapt Kapustin--Witten's definition of a $\BAA$-brane to this context. Suppose $Y^\vee$ is equipped with hyperk\"ahler structure and denote by $\Lambda$ the holomorphic symplectic $2$-form associated to the first K\"ahler structure. We say that a $\beta$-shifted sheaf $\Ggg = \{ G_i \}_{i \in I}$ {\it admits a $\beta$-shifted $\BAA$-brane structure} if the support of each of the $G_i$ is a Lagrangian sub-variety of $Y^\vee \times_V V_i$ with respect to $(\zeta_i^\vee)^*\Lambda$.

We will now describe a certain $\beta$-shifted morphism that will be crucial for our purposes. Recall that we denoted by $B' \subset B$ the locus of smooth spectral curves, and by $\Ss'$, the restriction of $\Ss$ and $\M$ to $B'$. Also, we denote by $\M'$ the restriction of $\M$ to $B'$, and by 
\[
\left ( \U', \boldsymbol{\Phi}' \right ) := \left . \left ( \U, \boldsymbol{\Phi} \right ) \right |_{\M'},
\]
the restriction of the universal bundle to $\M' \subset \M^\st$. Thanks to the spectral correspondence outlined in Section \ref{sc Hitchin fibration}, the existence of $(\U', \boldsymbol{\Phi}') \to \M'$ implies that, locally in the \'etale topology, there exists a $\pi_\M^*\beta$-twisted universal line bundle \[
\P' \to \Ss\times_{B'} \Jac^{\, \delta}_{B'}(\Ss'),
\]
satisfying $\U' = (\Id_{\Jac} \times p)_* \P'$, where $p$ is the projection \eqref{eq def p}. We can provide a specific, fine enough covering for it. 

\begin{proposition} \label{pr our etale covering}
There exists an \'etale covering $\{ V_i \to B' \}_{i \in I}$ of $B'$ such that 
\begin{equation} \label{eq our etale covering}
\left \lbrace Z_i : = \M' \times_{B'} V_i \to \M' \right \rbrace_{i \in I}
\end{equation}
is an \'etale covering of $\M'$ which is fine enough for %$\left ( \U', \boldsymbol{\Phi}' \right )$ and 
$\P'$.
\end{proposition}

\begin{proof}
%It is enough to construct a covering fine enough for $\P'$ since $\U' = (\Id_{\Jac} \times p)_* \P'$ and $\boldsymbol{\Phi}'$ is given by the multiplication under the restriction to $\P'$ of the tautological section. 
%
Recall that two spectral curves intersect on a divisor of length $2n(g-1)$. Let us pick a smooth spectral curve $S_i$ and set $B'_i$ to be the open subset of $B'$ given by those curves $S_b$ that intersect $S_i$ in $2n(g-1)$ different points, {\it i.e.} those curves giving a reduced intersection divisor $S_i \cap S_b$. Chose a collection $I$ of spectral curves in such a way that the union of all $B'_i$ covers $B'$. One can easily see that for every smooth spectral curve $S_b$ there exist another one $S_i$ such that $S_b \cap S_i$ is a reduced divisor; this guarantees the existence of a covering with the desired properties.

%it is guaranteed the existence of such covering given by $I$.

Inside $\Tot(K_X) \times B'_i$, we consider the intersection
\[
V_i := (S_i \times B'_i) \cap \Ss |_{\Tot(K_X) \times B'_i}.
\]
Observe that, by construction, this provides an \'etale morphism $V_i \to B'$ with image $B'_i$. This gives the \'etale covering $\{V_i \to B' \}_{i \in I}$. Denote
\begin{equation} \label{eq Ss'_i}
\Ss'_i := \Ss' \times_{B'} V_i
\end{equation}
and observe that $\Ss'_i \to V_i$ has naturally a section since $V_i$ embeds into $\Ss|_{B'_i}$. It then follows that $\Jac_{V_i}^\delta(\Ss'_i)$ exists and it is equipped with a universal line bundle $P'_i \to \Ss'_i \times_{V_i} \Jac_{V_i}^\delta(\Ss'_i)$. 
\end{proof}

Recalling the Poincar\'e bundle $\Pp \to \Jac_{B'}^\delta(\Ss') \times_{B'} \Jac_{B'}^\delta(\Ss')^\vee$, we observe that the $\pi_\M^*\beta$-twisted universal line bundle $\P'$ defines naturally a $\pi_\M^*\beta$-shifted closed immersion
\[
\morph{\Ss'}{\Jac_{B'}^\delta(\Ss')^\vee}{s}{\P'|_{\{s\} \times \Jac_{B'}^\delta(\Ss'),}}{}{\iI}
\]
given by $\iI = \{ \imath_i : \Ss'_i \to \Jac^{\delta}_{V_i}(\Ss'_i)  \}$ is determined by $P'_i$.

For each open \'etale subset of Proposition \ref{pr our etale covering}, we can consider a diagram analogous to \eqref{sc duality of the Hitchin system on M'} 
\begin{equation} \label{sc duality of the Hitchin system on Z_i}
  \begin{gathered}
\xymatrix{
& \Jac_{V_i}^{\, \delta}(\Ss'_i) \times_{V_i} \Jac_{V_i}^{\, \delta}(\Ss'_i)^\vee %\ar[d]^{\cong} &
%\\
%& \Jac_{V_i}^{\, \delta}(\Ss'_i) \times_{B'} \Jac_{V_i}^{\, \delta}(\Ss'_i)^\vee 
 \ar[rd]^{\check{\pi}} \ar[ld]_{\hat{\pi}} &
\\
Z_i \cong \Jac_{V_i}^{\, \delta}(\Ss'_i)  & & Z_i \cong \Jac_{V_i}^{\, \delta}(\Ss'_i)^\vee 
}    
  \end{gathered}
\end{equation}
and a Poincar\'e bundle
\[
\Pp_i \to \Jac_{V_i}^{\, \delta}(\Ss'_i) \times_{V_i} \Jac_{V_i}^{\, \delta}(\Ss'_i)^\vee.
\]

Proceeding as in \eqref{eq def FM} and \eqref{eq def FM*}, we define the Fourier--Mukai transforms $R\hat{\FFf}_i$ and $R\check{\FFf}_i$ replacing $\Pp$ by $\Pp_i$. Given a complex of $\beta$-twisted coherent sheaves $\F^\bullet$, define its $\beta$-twisted Fourier--Mukai transform $R\hat{\FFf}^{\beta}(\F^\bullet)$ as $\{ R\hat{\FFf}_i(F^\bullet_i) \to Z_i \}_{i \in I}$ and consider an analogous definition in the case of $R\check{\FFf}^{\beta}$. We say that a $\beta$-twisted coherent sheaf $\F = \{ F_i \to Z_i \}_{i \in I}$ is $\ell$-WIT if each of the $F_i$ is $\ell$-WIT with respect to $R\hat{\FFf}_i$. Similarly, a $\beta$-shifted coherent sheaf $\Ggg = \{ G_i \to Z_i \}_{i \in I}$ is $\ell$-WIT if each of the $G_i$ is $\ell$-WIT with respect to $R\check{\FFf}_i$. We denote by $\WWIT_\ell^{\beta}(\Jac_{B'}^{\, \delta}(\Ss'))$ the category of $\beta$-twisted $\ell$-WIT sheaves on $\Jac_{B'}^{\, \delta}(\Ss')$ and by $\wWIT_\ell^{\beta}(\Jac_{B'}^{\, \delta}(\Ss')^\vee)$ the category of $\beta$-shifted $\ell$-WIT sheaves on $\Jac_{B'}^{\, \delta}(\Ss')^\vee$.

In the following theorem we see that the $\beta$-twisted Fourier--Mukai transform relates both categories.

\begin{theorem} \label{tm FM with a gerbe}
The $\beta$-twisted Fourier--Mukai transform $R\hat{\FFf}^{\beta}$ induces an equivalence of categories
\[
R\hat{\FFf}^{\beta} : \WWIT_\ell^{\beta}(\Jac_{B'}^{\, \delta}(\Ss')) \stackrel{\cong}{\longrightarrow} \wWIT_{d-\ell}^{\beta}(\Jac_{B'}^{\, \delta}(\Ss')^\vee)
\]
with inverse $R\check{\FFf}^{\beta}$.
\end{theorem}

\begin{proof}
We first prove that the Fourier--Mukai transform of a $\ell$-WIT $\beta$-twisted coherent sheaf is a $(d-\ell)$-WIT $\beta$-shifted coherent sheaf. We observe that the following diagram 
\[
\xymatrix{
Z_{ij} \times_{V_{ij}} Z_{ij} \ar[d]_{\check{\pi}_{ij}} \ar[rrr]^{\Id_{Z_i} \times_{V_i} (\Id_{\Jac^\vee} \times_{B'} \zeta_i)} & & & Z_i \times_{V_i} Z_i \ar[d]^{\check{\pi}}
\\
Z_{ij} \ar[rrr]^{\Id_{\Jac^\vee} \times_{B'} \zeta_i} & & & Z_i
}
\]
commutes, being $\check{\pi}_{ij}$ the projection to the second factor. The diagram is also Cartesian as we naturally have $Z_{i} \times_{V_i} Z_{ij} \cong Z_{ij} \times_{V_{ij}} Z_{ij}$. Using the commutativity of the previous diagram and base change theorems, one can show that
\begin{align*}
(\Id_{\Jac^\vee} \times_{B'} \zeta_i)^* R\hat{\FFf}_i(F_i) & \cong (\Id_{\Jac^\vee} \times_{B'} \zeta_i)^* R\check{\pi}_* \left( \Pp_i \otimes \hat{\pi}^* F_i \right )
\\
& \cong R\check{\pi}_{ij,*}(\Id_{Z_i} \times_{V_i} (\Id_{\Jac^\vee} \times_{B'} \zeta_i))^*\left( \Pp_i \otimes \hat{\pi}^* F_i \right )
\\
& \cong R\check{\pi}_{ij,*} \left( \Pp_{ij} \otimes \hat{\pi}_{ij}^*(\Id_{\Jac^\vee} \times_{B'} \zeta_i)^* F_i \right )
\\
& \cong R\hat{\FFf}_{ij} \left( (\Id_{\Jac^\vee} \times_{B'} \zeta_i)^* F_i \right ),
\end{align*}
where $\Pp_{ij}$ is the Poincar\'e bundle over $Z_{ij} \times_{V_{ij}} Z_{ij} \cong \Jac_{V_{ij}}^\delta(\Ss_{ij}) \times_{V_{ij}} \Jac_{V_{ij}}^\delta(\Ss_{ij})$ and $R\hat{\FFf}_{ij}$ the corresponding Fourier--Mukai transform.

Let us recall that $\beta$-twisted sheaves must satisfy that
\[
(\Id_{\Jac^\vee} \times_{B'} \zeta_i)^* F_i \cong L_{ij} \otimes (\Id_{\Jac^\vee} \times_{B'} \zeta_j)^* F_j.
\]
Now we recall that one of the classical properties \cite{mukai} of the Fourier--Mukai transform states that
\[
R\check{\FFf}_{ij} \left( L_{ij} \otimes (\Id_{\Jac^\vee} \times_{B'} \zeta_j)^* F_j \right ) \cong \check{\sigma}_{ij,*} R\check{\FFf}_{ij} \left( (\Id_{\Jac^\vee} \times_{B'} \zeta_j)^* F_j \right).
\]
Then, 
\begin{align*}
(\Id_{\Jac^\vee} \times_{B'} \zeta_i)^* R\hat{\FFf}_i(F_i) & \cong R\check{\FFf}_{ij} \left( (\Id_{\Jac^\vee} \times_{B'} \zeta_i)^* F_i \right )
\\
& \cong \check{\sigma}_{ij,*} R\check{\FFf}_{ij} \left( (\Id_{\Jac^\vee} \times_{B'} \zeta_j)^* F_j \right)
\\
& \cong \check{\sigma}_{ij,*} (\Id_{\Jac^\vee} \times_{B'} \zeta_j)^* R\hat{\FFf}_j(F_j),
\end{align*}
so we see that the $R\hat{\FFf}^{\beta}(\F)$ is indeed a $\beta$-twisted sheaf.

Finally, since all the $R\hat{\FFf}_i$ are equivalence of categories of WIT sheaves with inverse $R\check{\FFf}_i$, the same holds between $\beta$-twisted and $\beta$-shifted WIT sheaves. 
\end{proof}

\subsection{The Dirac--Higgs bundle under Fourier--Mukai}

\label{sc FM Dirac-Higgs}

Recall that $\M' \subset \M^\st$ denotes the locus of the moduli space given by those Higgs bundles whose spectral curves are smooth. Let us consider the restriction to $\M'$ of the $\beta$-twisted Dirac--Higgs bundle
\[
\D' := \D|_{\M'}.
\]
In the remaining of the section we provide a description of $\D'$ in terms of a Fourier--Mukai transform. We shall consider the intersection of our family of smooth spectral curves $\Ss'$ with the zero section of the bundle $K_X\to X$, which is identified with $X \times \{ 0\}$,
\[
\Xi^0 := \Ss' \cap \left ( \left (X \times \{ 0 \} \right ) \times B' \right ).
\]
Note that for each $b = (b_1, \dots, b_n) \in B'$ one has that $S_b \cap X \times \{ 0 \}$ is the locus where $b_n = 0$, being $b_n \in H^0(K_X^{\otimes n})$. Recall that $\deg{K_X^{\otimes n}} = 2n(g-1)$. Then $\Xi^0$ is a finite cover over $B'$ of degree $2n(g-1)$,
\begin{equation} \label{eq finite cover}
  \begin{gathered}
\xymatrix{
\Xi^0 \ar[rr]^{2n(g-1) \, : \, 1} & & B'.
}    
  \end{gathered}
\end{equation}

Take an \'etale open subset $V_i \to B'$ and the family of curves $\Ss'_i \to V_i$ as defined in \eqref{eq Ss'_i} and consider the obvious projections occurring in the following commutative diagram, 
\begin{equation} \label{big diagram}
  \begin{gathered}
\xymatrix{
& \Jac_{V_i}^{\, \delta}(\Ss'_i) \times_{V_i} \Ss'_i \ar[rd]^{\widetilde{\pi}_S} \ar[ld]_{\widetilde{\pi}_{\Jac}} \ar[dd]^{(\Id_{\Jac} \times p)} &
\\
\Jac_{V_i}^{\, \delta}(\Ss'_i) \ar[dd]_{\Id_{\Jac}} & & \Ss'_i \ar[dd]^{p_i}
\\
& \Jac_{V_i}^{\, \delta}(\Ss'_i) \times X \ar[rd]^{\pi_{V \times X}} \ar[ld]_{\pi_{\Jac}} \ar[dd]^{\pi_X} &
\\
\Jac_{V_i}^{\, \delta}(\Ss'_i) & & V_i \times X \ar[ld]^{r_i}\\
& X &
}
  \end{gathered}
\end{equation}
Let us denote by $\widetilde{p}_i : \Ss'_i \to X$ the composition of $p_i$ with the obvious projection $r_i$ onto the second factor. Observe that the bundle $\widetilde{p}_i^*K_X \to \Ss'_i$ has a tautological section that we denote by $\lambda_i$. We observe that $\lambda_i$ vanishes at 
\[
\Xi^0_i := \Xi^0 \times_{B'} V_i. 
\]

Using \eqref{eq DD in terms of UU}, we give a description of $\D'$ which generalizes the fibrewise description given in \cite[Section 7]{hitchin_char}. Recall the $\beta$-shifted closed embedding $\iI : \Ss' \hookrightarrow \M'$ defined in Section \ref{sc FM for twisted vector bundles}.

\begin{proposition} \label{pr DD in terms of PP}
%Let us denote by $\widetilde{p} : \Ss \to X$ the composition of $p : \Ss \to B \times X$ with the obvious projection $r : B \times X \to X$. 

Consider the $\beta$-shifted push-forward under $\iI$,
\[
\check{\Ddd}' := \iI_* \left (\widetilde{p}^* K_X \otimes \Oo_{\Xi^0} \right).
\]
Then, $\check{\Ddd}' $ is a $\beta$-shifted $0$-WIT sheaf on $\Jac_{B'}^{\delta}(\Ss)^\vee$ and
\[
\D' \cong R \check{\FFf}^{\beta} \left ( \check{\Ddd}'  \right).
\]
\end{proposition}

\begin{proof}
We work locally over the \'etale open subset $Z_i = V_i \times_{B'} \M'$. Starting from \eqref{eq DD in terms of UU} and the isomorphism \eqref{eq M = Jac_B of S}, note that
$$ D'_i  \cong \RR^1 \pi_{\M, i, *} \left ( U'_i \stackrel{\Phi}{\longrightarrow} U'_i \otimes \pi_X^* K_X \right )
\cong \RR^1 \pi_{\Jac, *} \left (U'_i \stackrel{\Phi}{\longrightarrow} U'_i \otimes \pi_X^* K_X \right ). $$
Next, recalling the relation between the universal bundle from Section \ref{sc gerbes} and the Poincar\'e bundle described in Section \ref{sc FM for twisted vector bundles}, and making use of the projection formula and base change theorems for the various morphisms in the diagram \eqref{big diagram}, we obtain 
\begin{align*}
D'_i & \cong \RR^1\pi_{\Jac, *} \left ( (\Id_{\Jac} \times p_i)_* P'_i \stackrel{\Phi}{\longrightarrow} (\Id_{\Jac} \times p_i)_* P'_i \otimes \pi_X^* K_X \right )
\\
& \cong \RR^1 \pi_{\Jac, *} \left( (\Id_{\Jac} \times p_i)_* \left ( P'_i \stackrel{(\Id_{\Jac} \times \lambda_i)}{\longrightarrow} P'_i \otimes (\Id_{\Jac} \times p_i)^* \pi_X^* K_X \right ) \right)
\\
& \cong \RR^1 \widetilde{\pi}_{\Jac, *} \left ( P'_i \stackrel{(\Id_{\Jac} \times \lambda_i)}{\longrightarrow} P'_i \otimes \widetilde{\pi}_S^* \widetilde{p}^*_i K_X \right )
\\
& \cong \RR^1 \widetilde{\pi}_{\Jac, *} \left ( P'_i \otimes  \widetilde{\pi}_S^*( \Oo_{\Ss'_i} \stackrel{\lambda_i}{\longrightarrow} \widetilde{p}_i^* K_X ) \right)
\\
& \cong R^0\widetilde{\pi}_{\Jac, *}\left (\coker (\Id_{P_i'} \otimes \widetilde{\pi}^*_S \lambda_i) \right )
\\
& \cong  R^0 \widetilde{\pi}_{\Jac, *} \left ( P'_i \otimes  \widetilde{\pi}_S^* (\widetilde{p}^*_i K_X \otimes \Oo_{\Xi^0_i}) \right ), 
\end{align*}
where for the previous to last equality, we have used the (vertical) spectral sequence 
\begin{align*}
0 \longrightarrow R^1\widetilde{\pi}_{\Jac, *}\left (\ker (\Id_{P_i'} \otimes \widetilde{\pi}^*_S \lambda_i) \right ) \longrightarrow  \RR^1 \widetilde{\pi}_{\Jac, *} & \left ( P'_i \otimes  \widetilde{\pi}_S^*( \Oo_{\Ss'_i} \stackrel{\lambda_i}{\longrightarrow} \widetilde{p}_i^* K_X ) \right)  
\\
& \longrightarrow R^0\widetilde{\pi}_{\Jac, *}\left (\coker (\Id_{P_i'} \otimes \widetilde{\pi}^*_S \lambda_i) \right ) \longrightarrow 0,  
\end{align*}
and the fact that $\lambda_i$ is an embedding.

Recall that, by the definition of $\iI = \{ \imath_i : \Ss'_i \hookrightarrow \Jac_{V_i}^{\delta}(\Ss'_i)^\vee \}$, one has that the restriction of the Poincar\'e bundle $\Pp^*_i \to \Jac_{V_i}^{\delta}(\Ss'_i) \times_{V_i} \Jac_{V_i}^{\delta}(\Ss'_i)^\vee$ to the image of the embedding $\Id_{\Jac} \times \imath_i$, coincides with $P'_i$. Using this, the projection formula, and base change theorems on the diagram \eqref{big diagram}, we have that 
\begin{align*}
D'_i \cong & R^0 \widetilde{\pi}_{\Jac, *} \left ( P'_i \otimes  \widetilde{\pi}_S^* (\widetilde{p}_i^* K_X \otimes \Oo_{\Xi^0_i}) \right ) 
\\
\cong & R^0 \widetilde{\pi}_{\Jac, *} \left ((\Id_{\Jac} \times \imath_i)^* \Pp^*_i \otimes  \widetilde{\pi}_{S}^* (\widetilde{p}_i^* K_X \otimes \Oo_{\Xi^0_i}) \right ) 
\\
\cong & R^0 \hat{\pi}_*  R^0 \left ( (\Id_{\Jac} \times \imath_i)_* (\Id_{\Jac} \times \imath_i)^* \Pp^*_i \otimes \widetilde{\pi}_S^* (\widetilde{p}_i^* K_X \otimes \Oo_{\Xi^0_i}) \right )
\\
\cong & R^0 \hat{\pi}_* \left (\Pp^*_i \otimes R^0 (\Id_{\Jac} \times \imath_i)_* \widetilde{\pi}_S^* \left (\widetilde{p}_i^* K_X \otimes \Oo_{\Xi^0_i} \right ) \right )
\\
\cong & R^0 \hat{\pi}_* \left (\Pp^*_i \otimes  \check{\pi}^* R^0 \imath_{i,*} \left (\widetilde{p}_i^* K_X \otimes \Oo_{\Xi^0_i} \right ) \right )
\\
\cong & R^0 \check{\FFf}_i \left (R^0\imath_{i,*} \left (\widetilde{p}_i^* K_X \otimes \Oo_{\Xi^0_i} \right ) \right ).
\end{align*}
Since $\imath_i$ is a closed embedding, one has that
$R^0\imath_{i,*}(\widetilde{p}_i^* K_X \otimes \Oo_{\Xi^0_i})$
coincides with $\imath_{i,*}(\widetilde{p}_i^* K_X \otimes
\Oo_{\Xi^0_i})$. The support of $(\widetilde{p}_i^* K_X \otimes
\Oo_{\Xi^0_i})$ is $\Xi^0_i$, which is a finite $2n(g - 1)$-cover of
$V_i$, and so is $\imath_i(\Xi^0_i)$ which is the support of
$\imath_{i,*}(\widetilde{p}_i^* K_X \otimes
\Oo_{\Xi^0_i})$. Therefore, $R \check{\FFf}_i
(\imath_{i,*}(\widetilde{p}_i^* K_X \otimes \Oo_{\Xi^0_i}))$ is a
complex supported in degree $0$. 
\end{proof}

After Proposition \ref{pr DD in terms of PP} and equation \eqref{eq
  inverse of FM}, it is possible to study the Fourier--Mukai transform
of $\D'$. 

\begin{corollary} \label{co description of check D}
The Dirac--Higgs bundle $\D'$ is a $d$-WIT $\beta$-twisted sheaf,
where $d = 1 + n^2(g - 1)$, and   
\[
R \hat{\FFf}^{\beta} \left ( \D' \right ) \cong \check{\Ddd}'.
\]
with
\[
\supp \left ( \check{\Ddd}' \right ) = \iI \left (\Xi^0 \right ).
\]
\end{corollary}

\subsection{Nahm branes under Fourier--Mukai}

\label{sc Fourier-Mukai on Nahm transform}

We finished Section \ref{sc high rank Nahm} showing that the Nahm transform is a $\beta_m$-twisted $\BBB$-brane on the moduli space of stable Higgs bundles $\M_m^\st$ (or on a dense open subset in the case $n = m$) which called {\it Nahm brane associated to $\Ee$}. Using the formalism that we presented in Section \ref{sc mirror}, it is then natural to ask how the Nahm branes transform under mirror symmetry, which we address in this section.

Assume for simplicity that $n \neq m$ and that $\Ee$ is a stable Higgs bundle of rank $n$ and $0$ degree with reduced spectral curve $S_\Ee$. As $n \neq m$, the rank $m$ Nahm transform $\widehat{\boldsymbol{\Ee}}_m$ of $\Ee$ is a $\beta_m$-twisted $\BBB$-brane over $\M^\Ee_m = \M^\st_m$. Also, recall that the smooth locus of the Hitchin fibration is contained in the stable locus, $\M'_m \subset \M_m^\st$ and denote by 
\[
\widehat{\boldsymbol{\Ee}}'_m = \widehat{\boldsymbol{\Ee}}_m|_{\M'_m}
\]
the restriction of the rank $m$ Nahm transform. In this section we perform the Fourier--Mukai transform of $\widehat{\boldsymbol{\Ee}}'_m$. 

Recall the family $\Ss'_m \to B'_m$ of smooth spectral curves of rank $m$. Having in mind that $S_\Ee$ is reduced by hypothesis, consider as well the constant family of reduced curves
\[
\Ss_\Ee := S_\Ee \times B'_m.
\]
Define also the family of curves inside $\Tot(K_X)$ parametrized by $B'_m$,
\[
\Sigma_m^\Ee := \sigma\left ( \Ss_\Ee \times_{X \times B'_m} \Ss'_m \right ),
\]
where $\sigma$ denotes here the fibrewise sum in $r_m^* K_X \to X \times B'_m$, where $r_m : B_m' \times X \to X$ is the projection onto the second factor. After Proposition \ref{pr description of Sigma_Phi}, one has the following commutative diagram for families given by the obvious projections
\begin{equation} \label{eq commutative diagram families}
  \begin{gathered}
    \xymatrix{
 & \Sigma_m^\Ee \ar[ld]_{\widetilde{q}_\Ee}^{m:1} \ar[rd]^{q_m}_{n:1} \ar[dd]^{p_m^\Ee}_{nm:1} &
\\ 
\Ss_\Ee \ar[rd]_{\widetilde{p}_\Ee}^{n:1} &  & \Ss'_m \ar[ld]^{p_m}_{m:1}
\\
 & X \times B'_m &
}
  \end{gathered}
\end{equation}
where $p_m^\Ee$ is a $nm$-cover, $\widetilde{p}_\Ee = (p_\Ee \times \Id_{B'_m})$ and $q_m$ are $n$-covers, and $p_m$ and $\widetilde{q}_\Ee$ are $m$-covers. 

Noting that $S_\Ee\subset \Tot(K_X)$, consider $(-1)$ to be the (additive) inversion along the fibres of $K_X$ and consider the family of curves $-\Ss_\Ee$. Define one more family of curves over $B'_m$,
\[
\Xi_m^\Ee := \Ss'_m \cap -\Ss_\Ee.
\]
Note that $\Xi_m^\Ee$ equals $\Sigma_m^\Ee \cap (X \times \{ 0 \})
\times B'_m$. Since $\Sigma_m^\Ee$ is a family of spectral curves of
the form $S_{\Ee \otimes \Ff}$, by \eqref{eq finite cover} we have
that $\Xi_m^\Ee$ is a $2nm(g-1)$-cover of $B'_m$, 
\[
  \xymatrix{
    \Xi_m^\Ee \ar[rr]^{2nm(g-1) \, : \, 1} & & B'_m.
  }
\]

Given an \'etale open subset $V_{m,i} \to B'_m$, we consider $\Ss'_{m,i} := \Ss_m \times_{B'} V_i$, $\Ss'_{\Ee,i} := \Ss_{\Ee} \times_{B'} V_i$, $\Sigma_{m,i}^\Ee := \Sigma_m^\Ee \times_{B'} V_i$ and $\Xi_{m,i}^\Ee := \Xi_m^\Ee \times_{B'} V_i$. We consider as well the corresponding lifts of the morphisms appearing in \eqref{eq commutative diagram families}. Over $\Ss_{\Ee,i} \to V_{m,i}$, consider the constant family of rank $1$ torsion free sheaves $\Ll_{\Ee,i}$ determined by $L_\Ee$ and observe that $(-1)^*\Ll_{\Ee,i}$ is supported on $-\Ss_{\Ee,i}$, hence one can restrict this sheaf to $\Xi^\Ee_{m,i}$. 

The morphism $\widetilde{p}_{m,i} : \Ss'_{m,i} \to X$ obtained as the composition of $p_{m,i}$ with the obvious projection $r_{m,i}$ onto the second factor. The bundle $\widetilde{p}_{m,i}^*K_X \to \Ss'_{m,i}$ has a tautological section $\lambda_{m,i}$ which vanishes at $\Xi^\Ee_{m,i}$.

We can now provide a result analogous to Proposition \ref{pr DD in terms of PP} for $\widehat{\E}'_m:= \widehat{\E}_m|_{\M'_m}$. We recall the $\beta_m$-shifted closed embedding $\iI_m : \Ss'_m \hookrightarrow \M'_m$ defined in Section \ref{sc FM for twisted vector bundles}.

\begin{theorem} \label{tm FM of hat E}
Let $\Ee = (E, \varphi)$ be a stable Higgs bundle with spectral data $L_\Ee \to S_\Ee$ such that $S_\Ee$ is reduced. Consider the $\beta_m$-shifted push-forward under $\iI_m$,
\[
\widecheck{\Eee}'_m := \iI_{m,*} \left (\widetilde{p}_m^*K_X \otimes (-1)^*\Ll_\Ee|_{\Xi^\Ee_m} \right ).
\]
Then, $\widecheck{\Eee}'_m$ is a $\beta_m$-shifted $0$-WIT sheaf on $\Jac_{B'}^{\delta}(\Ss)^\vee$ and
\[
\widehat{\E}'_m \cong R \check{\FFf}^{\beta_m}_m \left ( \widecheck{\Eee}'_m \right ).
\]\end{theorem}

\begin{proof}
Recalling Proposition \ref{pr Nahm transform a la Hausel}, we proceed as in the proof of Proposition \ref{pr DD in terms of PP} with the appropriate modifications. We have 
\begin{align*}
\widehat{E}'_{m,i} & \cong  \RR^1 \pi_{\Jac, *}  \left (U'_{m,i} \otimes \pi_{X}^* E \stackrel{\Phi_{m,i} \otimes \Id + \Id \otimes \pi_{X}^*\varphi}{\longrightarrow} U'_{m,i} \otimes  \pi_{X}^* E  \otimes \pi_X^* K_X \right )
\\
& \cong \RR^1\pi_{\Jac, *} \left ((\Id_{\Jac} \times p_{m,i})_* P'_{m,i} \otimes \pi_{X}^* E \stackrel{\Phi_{m,i} \otimes \Id + \Id \otimes \pi_{X}^*\varphi}{\longrightarrow} (\Id_{\Jac} \times p_{m,i})_* P'_{m,i} \otimes \pi_X^*(E \otimes  K_X) \right )
\\
& \cong \RR^1 \pi_{\Jac, *} \left((\Id_{\Jac} \times p_{m,i})_* \left ( P'_{m,i} \otimes \widetilde{\pi}^*_S \widetilde{p}_{m,i}^* E \stackrel{(\Id_{\Jac} \times \lambda_{m,i}) \otimes \Id + \Id \otimes \widetilde{\pi}^*_S \widetilde{p}_{m,i}^*\varphi}{\longrightarrow} P'_{m,i} \otimes \widetilde{\pi}^*_S \widetilde{p}_{m,i}^* (E \otimes K_X) \right ) \right)
\\
& \cong \RR^1 \widetilde{\pi}_{\Jac, *} \left (P'_{m,i} \otimes \widetilde{\pi}_S^* \left ( \widetilde{p}_{m,i}^* E \stackrel{\Id \otimes \lambda_{m,i} + \widetilde{p}_{m,i}^*\varphi \otimes \Id}{\longrightarrow} \widetilde{p}_{m,i}^*E \otimes \widetilde{p}_{m,i}^*K_X \right ) \right ).
\end{align*}
Using a (vertical) spectral sequence similar to the one that appears in Proposition \ref{pr DD in terms of PP}, one can show
\[
\widehat{E}'_{m,i} \cong R^0 \widetilde{\pi}_{\Jac, *} \left (P'_{m,i} \otimes \widetilde{\pi}_S^* \left (\widetilde{p}_{m,i}^*K_X \otimes (-1)^*\Ll_{\Ee,i}|_{\Xi_{m,i}^\Ee} \right ) \right ),
\]
since $\Id \otimes \lambda_{m,i} + \widetilde{p}_{m,i}^*\varphi \otimes \Id$ is again injective.

Next, using the description of the Poincar\'e bundle $\P_m$ in terms of $\Pp_m$ given in Section \ref{sc gerbes}, the projection formula and base change theorems for the various morphisms of diagram \eqref{big diagram}, one obtains 
\[
\widehat{\E}'_m \cong R^0 \check{\FFf}_m^{\beta_m} \left (R^0\imath_{m,i,*}\left (\widetilde{p}_m^*K_X \otimes (-1)^*\Ll_\Ee|_{\Xi_m^\Ee} \right ) \right ).
\]
Finally, as in the last part of the proof of Proposition \ref{pr DD in terms of PP}, due to the fact that $\Xi_{m,i}^\Ee$ is a finite $2nm(g - 1)$-cover of $B'_m$, $R \check{\FFf}_m^{\beta_m} \left (\imath_{m,i,*}\left (\widetilde{p}_{m,i}^*K_X \otimes (-1)^*\Ll_\Ee|_{\Xi_{m,i}^\Ee} \right ) \right )$ is a complex supported in degree $0$.
\end{proof}

Thanks to Theorem \ref{tm FM of hat E} and \eqref{eq inverse of FM} we can describe the Fourier--Mukai transform restricted to the smooth locus of the Hitchin fibration $\M'_m \subset \M_m$.

\begin{corollary} \label{co description of check Ee}
Let $d_m = 1 + m^2(g - 1)$. The $\beta_m$-twisted bundle $\widehat{\E}'_m$ is a $\beta_m$-twisted $d_m$-WIT sheaf satisfying 
\[
R \hat{\FFf}_m^{\beta_m} \left ( \widehat{\E}'_m \right ) \cong \widecheck{\Eee}'_m,
\]
where
\[
\supp \left ( \widecheck{\Eee}'_m \right ) = \iI_m \left (- \Xi^\Ee_m \right ).
\]
\end{corollary}

We study now its support.

\begin{theorem} \label{tm Xi^Ee Lagrangian}
The support of $\widecheck{\Eee}'_m$ is a $\beta_m$-shifted $2nm(g-1)$-section of the Hitchin fibration and it is Lagrangian with respect to $\Lambda_m^1$.
\end{theorem}

\begin{proof}
Recall the \'etale covering $\{ V_{m,i} \to B'_m  \}_{i \in I}$ of the locus of smooth spectral curves which is fine enough for the gerbe $\beta_m$. This induces an \'etale covering $\left \lbrace z_i : \M'_m \times_{B'} V_{m,i} \to \M'_m \right \rbrace_{i \in I}$. Since $\iI_{m}$ is a $\beta_m$-shifted embedding and $\Xi^\Ee_m$ is a $2nm(g-1)$-cover of $B'_m$, we have that $\iI_{m}(-\Xi^\Ee_m)$ is a $\beta_m$-shifted $2nm(g-1)$-section of $\M'_m \to B'_{m}$. It remains to proof that it is Lagrangian with respect to $\Lambda^1_{m,i} := z_i^*\Lambda^1_m$. 

Recalling Corollary \ref{co description of check Ee}, we now address the proof that
$\iI_m(-\Xi^\Ee_m)$ is Lagragian with respect to $\Lambda_m^1$. Recall from Section \ref{sc Mm hyperkahler}, that $\Lambda_m^1$ is defined as the exterior derivative $d\theta$ of a certain $1$-form $\theta$. Then, we see that $\imath_{m,i}(-\Xi^\Ee_{m,i})$ is Lagrangian with respect to $\Lambda^1_{m,i}$ if and only if $\theta_i := z_i^*\theta$ is a constant 1-form along $\imath_{m,i}(-\Xi^\Ee_{m,i})$ for all $i \in I$. Since $\imath_{m,i}$ is an embedding, it suffices to prove that $\imath_{m,i}^*\theta_i|_{-\Xi^\Ee_{m,i}}$ is constant.

We recall the definition of $\theta$. We recall that $\M'_m \subset \M^\st_m$ so all the points are smooth and represented by the stable Higgs bundle $\Ee = (E, \varphi)$, the tangent space is $T_{\Ee} \M'_m = \HH^1(C^\bullet_\Ee)$, which comes naturally equipped with the map $t: \HH^1(C^\bullet_\Ee) \to H^1(X,\End(E))$. By Serre duality, the Higgs field $\varphi \in H^0(\End(E) \otimes K_X)$ is an element of the dual space of $H^1(\End(E))$ and recall that we defined $\theta(v) = \langle \varphi, t(v) \rangle$, for each $v \in \HH^1(C^\bullet_\Ee)$.

We now study the description of $\theta_i$ over $\Jac_{V_{m,i}}^{\,
  \delta_m}(\Ss'_{m,i})$. By the spectral correspondence, given the
spectral data $L \to S_{m,b}$, one has that $E = (p_b)_*L$ and
$\varphi = (p_b)_*\lambda_b$ where $\lambda_b : L \to L \otimes
p_b^*K_X$ is given by tensoring by the restriction to $S_{m,b}$ of the
tautological section $\lambda : \Tot(K_X) \to p^*K_X$. Note that
$p_b^*K_X$ is a sub-sheaf of the canonical bundle $K_{S_{m,b}}$, then
$\lambda_b$ gives naturally an element of $H^0(S_{m,b},
K_{S_{m,b}})$. The isomorphism $\M'_{m,i} \cong \Jac_{V_{m,i}}^{\,
  \delta_m}(\Ss'_{m,i})^\vee \cong \Jac_{V_{m,i}}^{\,
  \delta_m}(\Ss'_{m,i})$, given by the push-forward under $p_b$ and
autoduality, provides as well the isomorphism between
$\HH^1(C^\bullet_\Ee)$ and $\Ext^1_{\Tot(K_X)}(L,L)$ and between
$H^1(X,\End(E))$ and $\Ext^1_{S_{m,b}}(L,L) \cong H^1(S_{m,b},
\Oo_{S_{m,b}})$. Then, we can express $\theta_i(v)$ to be $\langle
\lambda_b, t'(v) \rangle$ given by Serre duality, where now $v \in T_L
\Jac_{V_{m,i}}^{\, \delta_m}(\Ss'_{m,i}) \cong \Ext^1_{\Tot(K_X)}(L,
L)$, the section $\lambda_b \in H^0(S_{m,b}, K_{S_{m,b}})$ is defined
by the restriction of the tautological section $\lambda : \Tot(K_X)
\to p^*K_X$ to $S_{m,b} \subset \Tot(K_X)$, and 
\[
  t' : \Ext^1_{\Tot(K_X)}(L,L) \to \Ext^1_{S_{m,b}}(L,L) \cong
  H^1(S_{m,b}, \Oo_{S_{m,b}}) 
\]
is the projection to those deformations that preserve the support.

Note that, for every $L_1, L_2 \in \Jac(S_{m,b})$, one has naturally that 
\[
\Ext^1_{S_{m,b}}(L_1, L_1) \cong H^1(S_{m,b}, \Oo_{S_{m,b}}) \cong \Ext^1_{S_{m,b}}(L_2, L_2).
\]
We observe that $\theta_i$ is a 1-form which is constant along the fibres $\Jac^{\delta_m}(S_{m,b})$. On the other hand, the 1-form $\imath_{m,i}^*\theta_i$ in $\Ss'_{m,i}$ depends on the embedding $d\imath_{m,i} : T_s S_{m,b} \hookrightarrow T_{\Oo(s - s_i)}\Jac^{\delta_m}(S_{m,b}) \cong H^1(S_{m,b}, \Oo_{S_{m,b}})$. Recall that $\imath_{m,i}$ sends the point $s \in S_{m,b}$ to the line bundle whose meromorphic sections have pole at $s \in S_{m,b}$ and a zero at $s_i$. Since Serre duality $\langle \cdot, \cdot \rangle : H^0(K_{S_{m,b}}) \times H^1(\Oo_{S_{m,b}}) \to \CC$ sends $\langle \lambda , \xi \rangle$ to the sum of residues of the meromorphic differential $\lambda \xi$, one has that 
\[
\imath_{m,i}^* \theta_i |_s \cong \imath_{m,i}^* \langle \lambda_b , \cdot \rangle |_s \cong \lambda_b(s).
\]
So, $\imath_{m,i}^* \theta_i$ is the one form defined by the tautological section $\lambda : \Tot(K_X) \to p^*K_X$.

Obviously, the tautological section $\lambda$ restricted to $X \times
\{ 0\} \subset \Tot(K_X)$, is the zero section. Recall that we have
defined $-\Xi^\Ee_{m,i}$ as the intersection of $\Ss_{\Ee,i}$ and
$\Ss'_{m,i}$ inside $\Tot(K_X)$. But this is equivalent to the
intersection of $\Sigma_{m,i}^{\Ee^*}$ with $X \times \{ 0
\}$. Therefore, $\lambda$ is constantly $0$ along $-\Xi^\Ee_{m,i}$,
that is $\imath_{m,i}^* \theta_i|_{\Xi^\Ee_{m,i}} = 0$, and this
concludes the proof.
\end{proof}

We finish placing this statement in the context of mirror symmetry.

\begin{corollary} 
The $\beta_m$-shifted coherent sheaf $\widecheck{\Eee}'_m$ admits a
$\beta_m$-shifted $\BAA$-brane structure. 
\end{corollary}

%%%%%%%%%%%%%%%%%%%%%
% References
%%%%%%%%%%%%%%%%%%%%%

\end{document}